\documentclass[11pt,reqno,a4paper]{amsart}

\usepackage{etoolbox}
\usepackage{amsmath}
\usepackage{enumerate}
\usepackage{amssymb}
\usepackage{amscd}
\usepackage{amsthm}
\usepackage{amsfonts}
\usepackage{graphicx}
\usepackage[all,cmtip]{xy}

\patchcmd{\subsection}{-.5em}{.5em}{}{}
\patchcmd{\subsubsection}{-.5em}{.5em}{}{}

\usepackage{enumitem}

\usepackage{kmath,kerkis}
\usepackage[T1]{fontenc}

\makeatother
\usepackage{hyperref}

\numberwithin{equation}{section}

\newcommand{\cC}{\mathcal{C}}

\newcommand{\cF}{\mathcal{F}}

\newcommand{\cL}{\mathcal{L}}
\newcommand{\cM}{\mathcal{M}}

\newcommand{\cP}{\mathcal{P}}

\newcommand{\cS}{\mathcal{S}}

\newcommand{\bF}{\mathbb{F}}

\newcommand{\bN}{\mathbb{N}}

\newcommand{\bR}{\mathbb{R}}

\newcommand{\bZ}{\mathbb{Z}}

\newcommand{\ra}{\rightarrow}

\newcommand{\qand}{\quad \textrm{and} \quad}
\newcommand{\qqand}{\qquad \textrm{and} \qquad}

\newcommand\subsetsim{\mathrel{%
\ooalign{\raise0.2ex\hbox{$\subset$}\cr\hidewidth\raise-0.8ex\hbox{\scalebox{0.9}{$\sim$}}\hidewidth\cr}}}
\newcommand{\eps}{\varepsilon}

\DeclareMathOperator{\supp}{supp}

\DeclareMathOperator{\Aff}{Aff}

\theoremstyle{theorem}
\newtheorem{theorem}{Theorem}[section]
\newtheorem{corollary}[theorem]{Corollary}

\newtheorem{lemma}[theorem]{Lemma}
\newtheorem{scholium}[theorem]{Scholium}

\theoremstyle{definition}
\newtheorem{definition}[theorem]{Definition}
\newtheorem{remark}[theorem]{Remark}
\newtheorem{example}[theorem]{Example}

\begin{document}

\title{Product set phenomena for measured groups}

\author{Michael Bj\"orklund}

\address{Department of Mathematics, Chalmers, Gothenburg, Sweden}
\email{micbjo@chalmers.se}


\subjclass[2010]{}

\date{}


\maketitle

\begin{abstract}
Following the works of Furstenberg and Glasner on stationary means, we strengthen and extend in this paper 
some recent results by Di Nasso, Goldbring, Jin, Leth, Lupini and Mahlburg on piecewise syndeticity of product 
sets in countable \textsc{amenable} groups to general countable measured groups. We point out several fundamental differences between the behavior of products of "large" sets in Liouville and non-Liouville measured groups. 

As a (very) special case of our main results, we show that if $G$ is a free group of finite rank, and $A$ and $B$ are "spherically large" subsets of $G$, then there exists a finite set $F \subset G$ such that $AFB$ is thick. The position
of the set $F$ is curious, but seems to be necessary; in fact, we can produce \emph{left thick} sets $A, B \subset G$ 
such that $B$ is "spherically large", but $AB$ is \emph{not} piecewise syndetic. On the other hand, if $A$ is spherically
large, then $AA^{-1}$ is always piecewise syndetic \emph{and} left piecewise syndetic. However, contrary to what happens for amenable groups,  $AA^{-1}$ may fail to be syndetic.
The same phenomena occur for many other (even amenable, but non-Liouville) measured groups.

Our proofs are based on some ergodic-theoretical results concerning stationary 
actions which should be of independent interest.  
\end{abstract}

\section{Introduction}

Let $G$ be a countable group. Given two sets $A, B \subset G$, we define their \textbf{product set} $AB$ by
\[
AB = \big\{ ab \, : \, a \in A, \: b \in B \big\}.
\]
Let $(\beta_n)$ be a sequence of probability measures on $G$, and define for $B \subset G$, the 
\textbf{upper} and \textbf{lower asymptotic density} of $B$ with respect to $(\beta_n)$ by
\begin{equation}
\label{defbetan}
\beta^*(B) = \varlimsup_n \beta_n(B) 
\qand
\beta_*(B) = \varliminf_n \beta_n(B).
\end{equation}
We say that a set $C \subset G$ is 
\begin{itemize}
\item \textbf{thick} if for every finite set $L \subset G$, there is $g \in G$ such that $Lg \subset C$.
\item \textbf{left thick} if for every finite set $L \subset G$, there is $g \in G$ such that $gL \subset C$.\\

\item \textbf{syndetic} if there exists a finite set $F \subset G$ such that $FC = G$. 
\item \textbf{left syndetic} if there exists a finite set $F \subset G$ such that $CF = G$. \\

\item \textbf{piecewise syndetic} if there exists a finite set $F \subset G$ such that $FC$ is thick.
\item \textbf{piecewise left syndetic} if there exists a finite set $F \subset G$ such that $CF$ is thick.
\end{itemize}

Of course, if $G$ is abelian, then all of the paired definitions above can be merged, and we can without
confusion talk about thick, syndetic and piecewise syndetic sets in $G$. However, as soon as $G$ has
at least one element with an \textsc{infinite} conjugation class, then all six notions above are distinct.

\subsection{Amenable groups}

A sequence $(F_n)$ of finite subsets of $G$ is said to be \textbf{F\o lner} if
\[
\forall\, s \in G, \quad \varlimsup_n \, \frac{|sF_n \Delta F_n|}{|F_n|} = 0,
\]
and $G$ is \textbf{amenable} if it admits a F\o lner sequence. Every countable group of sub-exponential growth,
as well as every countable solvable (in particular, abelian) group is amenable. On the other hand, every non-cyclic free 
group is non-amenable. \\

Suppose that $G$ is amenable and let $(F_n)$ be a F\o lner sequence in $G$. Define a sequence $(\beta_n)$ of
probability measures on $G$ by
\begin{equation}
\label{defbetan}
\beta_n(B) = \frac{|B \cap F_n|}{|F_n|}, \quad \textrm{for $B \subset G$}.
\end{equation}
We say that $B \subset G$ is \textbf{large with respect to $(F_n)$} if $\beta^*(B) > 0$, and \textbf{large} if it
is large with respect to \emph{some} F\o lner sequence in $G$. A classical result of F\o lner \cite{Fol} says
that whenever $B \subset G$ is a large set, then the difference set $BB^{-1}$ is syndetic (and thus left syndetic
as well by symmetry). This can be thought 
of as a discrete analogue of a classical and useful result by Steinhaus, which asserts that the difference 
set of any Borel set of the reals with positive Lebesgue measure contains an interval. \\

More recently, Jin proved in his influential paper \cite{Jin} that if $A, B \subset \bZ$ are two large sets, then
their product (sum) set is piecewise syndetic. Jin referred to this observation as the \emph{sumset phenomenon}.
A few years later, Beiglb\"ock, Bergelson and Fish \cite{BBF} extended the sumset phenomenon to products of 
large subsets in arbitrary countable amenable groups. Even more recently, Di Nasso, Goldbring, Jin, Leth, Lupini 
and Mahlburg in the papers \cite{BLAJ0} and \cite{BLAJ1} were able to "localize" the sumset phenomenon in any
countable amenable group $G$. More precisely, they prove in \cite{BLAJ1} that if $A \subset G$ is large, 
and $(F_n)$ is any F\o lner sequence in $G$, then, for every $B \subset G$ there exists a \emph{finite} set 
$F \subset G$ such that
\begin{equation}
\label{di1}
\beta^*\Big(\bigcap_{l \in L} l FAB\Big) \geq \beta^*(B), \quad \textrm{for every \emph{finite} set $L \subset G$},
\end{equation}
where $\beta^*$ is defined using the sequence $(\beta_n)$ in \eqref{defbetan}. In particular, if $\beta^*(B)$ is 
assumed to be positive, then this inequality implies that all of the possible intersections on the left hand side are non-empty, which readily shows that $FAB$ is thick, and thus $AB$ is piecewise syndetic. Moreover, for any 
$\eps > 0$, they construct a finite set $K \subset G$, which is independent of $B \subset G$ and the 
F\o lner sequence $(F_n)$ such that
\begin{equation}
\label{di2}
\beta_*\Big(\bigcap_{l \in L} l KAB\Big) \geq \beta_*(B) - \eps, \quad \textrm{for every \emph{finite} set $L \subset G$}.
\end{equation}
The proofs in the papers \cite{BLAJ0} and \cite{BLAJ1} combine elementary combinatorics and non-standard analysis. The aim of this paper is to outline
an alternative approach using ergodic theory, which has the advantage of not only providing a strengthening of the results mentioned above, but is flexible enough to also tackle the case of products of sets in arbitrary countable groups
which are "large" with respect to certain stationary (or harmonic) densities. This line of study was initiated by Furstenberg and Glasner in the works \cite{FG1} and \cite{FG2}, and it was continued by the author and Fish in \cite{BF3}. Before we present our main results in their 
most general setting, we specialize them first to products in free groups.

\subsection{Spherical densities in free groups}
\label{subsec:sphere}
Let $\bF_r$ denote the free group on a set of $r$ (free) generators, which we here denote by 
$S = \big\{a_1^{\pm 1},\ldots, a_r^{\pm 1} \big\}$. In particular, we have $\bF_1 \cong \bZ$, 
which is an abelian, and thus amenable, group. However, for $r \geq 2$, the group $\bF_r$ is \emph{not} amenable. \\

We set $S_o = \{e\}$ and for $k \geq 1$, we define $S_k = S^{k} \setminus S^{k-1}$. One recognizes $S_k$ as
the sphere of radius $k$ in the Cayley graph associated to $(\bF_r,S)$. We define the 
\textbf{upper} and \textbf{lower spherical density} of a subset $B \subset \bF_r$ by
\[
\overline{s}_r(B) = \varlimsup_n \, \frac{1}{n} \sum_{k=0}^{n-1} \frac{|B \cap S_k|}{|S_k|}
\qqand 
\underline{s}_{r}(B) = \varliminf_n \, \frac{1}{n} \sum_{k=0}^{n-1} \frac{|B \cap S_k|}{|S_k|}
\]
respectively. We say that $B \subset \bF_d$ is \textbf{spherically large} if $\overline{s}_r(B) > 0$.  As far as
we know, these densities have not been studied before, except when  $r = 1$, in which case we recover the 
classical and well-studied upper and lower asymptotic densities for a subset $B \subset \bZ$ defined by
\[
\overline{s}_1(B) = \varlimsup_n \, \frac{|B \cap [-n,n]|}{2n+1}
\qand
\underline{s}_1(B) = \varliminf_n \, \frac{|B \cap [-n,n]|}{2n+1}.
\]
We can now state the following analogue for spherical 
densities of the results by Di Nasso, Goldbring, Jin, Leth, Lupini and Mahlburg (for amenable groups) summarized in \eqref{di1} and \eqref{di2} above.

\begin{theorem}
\label{spherical}
Suppose that $A \subset \bF_r$ is spherically large. For every $B \subset \bF_r$, there exists a finite set 
$F \subset G$ such that for every finite set $L \subset G$, we have
\[
\overline{s}_r\Big(\bigcap_{l \in L} l AFB\Big) \geq \overline{s}_r(B).
\]
Furthermore, for every $\eps > 0$, there exists a finite set $K \subset G$ such that for \emph{every} $B \subset G$
and \emph{finite} set $L \subset G$, we have
\[
\underline{s}_r\Big(\bigcap_{l \in L} l AKB\Big) \geq \underline{s}_r(B) - \eps.
\]
\end{theorem}

Contrary to what happens for large sets in \emph{amenable} groups, the set $F$ appears \emph{between} $A$ 
and $B$, and not to the left of $AB$. In particular, Theorem \ref{spherical} does \emph{not} (at least not directly) 
imply that $AB$ is piecewise syndetic for all spherically large $A, B \subset \bF_r$. In fact, we do not know whether 
this is true. Our methods do not seem to be able to tackle this question. On the other hand, the following (piecewise) analogue of F\o lner's Theorem \emph{does} hold (see also Theorem \ref{piecewiseFol} below).

\begin{theorem}
\label{spherical2}
Suppose that $A \subset \bF_r$ is spherically large. Then the difference set $AA^{-1}$ is both piecewise syndetic
and piecewise left syndetic.
\end{theorem}

The word "piecewise" in Theorem \ref{spherical2} cannot be removed. Indeed, in the appendix of this paper, we construct, for every $0 < \alpha < 1$, a set $A \subset \bF_r$ with $\underline{s}_r(A) = \alpha$ such that 
$AA^{-1}$ is \textsc{not} syndetic. 

\subsection{Connection to stationary densities}

Let us now connect the discussion of spherical densities to the general theme of this paper. We let $\sigma_o = \delta_e$, where $e$ denotes the identity element in $\bF_r$, and for $k \geq 1$ we define the probability measure
$\sigma_k$ on $\bF_r$ by
\[
\sigma_k(B) = \frac{|B \cap S_k|}{|S_k|}, \quad \textrm{for $B \subset \bF_r$}.
\]
We note that $(\sigma_k)$ satisfies the so called \emph{Hecke recurrence relations}, namely
\begin{equation}
\label{hecke}
\sigma_1 * \sigma_k = \frac{1}{2r} \sigma_{k-1} + \big( 1 - \frac{1}{2r}\big) \sigma_{k+1}, \quad \textrm{for all $k \geq 1$}.
\end{equation}
In particular, if we adopt the convention that $\sigma_1^{*0} = \delta_e$, then every $\sigma_k$ can be written as a convex combination of convolution powers of $\sigma_1$. We define the sequence $(\beta_n)$ of probability measures on $\bF_r$ by
\begin{equation}
\label{sigmaseq}
\beta_n = \frac{1}{n} \sum_{k=0}^{n-1} \sigma_k, \quad \textrm{for $n \geq 1$},
\end{equation}
and note that $\beta^* = \overline{s}_r$ and $\beta_* = \underline{s}_r$, where $\beta^*$ and $\beta_*$ are defined
as in \eqref{defbetan}. If we think of $(\beta_n)$ as 
positive and unital functionals in $\ell^\infty(\bF_r)^{*}$, then Banach-Alaoglu's Theorem asserts that the
set $\cS_r$ of weak*-cluster points of $(\beta_n)$ is non-empty. One readily checks that every $\lambda \in \cS_r$
is positive and unital, and it follows from \eqref{hecke} that
\begin{equation}
\label{statlambda}
\sum_{s \in S} \sigma_1(s) \lambda(s^{-1}B) = \lambda(B), \quad \textrm{for all $B \subset \bF_r$},
\end{equation}
where we abuse notation and write $\lambda(B)$ to denote $\lambda(\chi_B)$, for the indicator function 
$\chi_B$ of $B$. Furthermore, we clearly have
\[
\overline{s}_r(B) = \sup\big\{ \lambda(B) \, : \, \lambda \in \cS_r \big\}
\qand
\underline{s}_r(B) = \inf\big\{ \lambda(B) \, : \, \lambda \in \cS_r \big\},
\]
for every subset $B \subset \bF_r$. \\

More generally, let $G$ be a countable group and let $p$ be a probability measure on $G$. We 
say that $p$ is 
\begin{itemize}
\item \textbf{admissible} if the support of $p$ generates $G$ as a semi-group.
\item \textbf{symmetric} if $p(s^{-1}) = p(s)$ for all $s \in G$. 
\end{itemize}
If $p$ is admissible, we refer to $(G,p)$ as a \textbf{measured group}. Note that $(\bF_r,\sigma_1)$ is a (symmetric)
measured group for every $r \geq 1$. \\

If $f \in \ell^\infty(G)$ and $s \in G$, we define
\[
(s \cdot f)(g) = f(s^{-1}g) \qand (f \cdot s)(g) = f(gs), \quad \textrm{for $g \in G$}.
\]
Let $\cM$ denote the weak*-compact and convex set of all positive and unital elements in $\ell^\infty(G)^*$. Such 
functionals are often called \textbf{means} on $G$, and they can be equivalently thought of as \emph{finitely additive}
probability measures on $G$ by writing $\lambda(B)$ for $\lambda(\chi_B)$, where $\chi_B$ denotes the indicator 
function on $B$. Note that the (left) $G$-action on $\ell^\infty(G)$ above induces a (left) $G$-action on $\cM$ by
\[
(s \cdot \eta)(f) = \eta(s^{-1} \cdot f), \quad \textrm{for $s \in G$ and $f \in \ell^\infty(G)$}.
\]
If $p$ is a probability measure on $G$ we write
\begin{equation}
\label{defGact}
f * p = \sum_{s \in G} p(s) (f \cdot s) \qand p * \eta = \sum_{s \in G} p(s) (s \cdot \eta),
\end{equation}
for all $f \in \ell^\infty(G)$ and $\eta \in \cM$, and we define
\begin{equation}
\label{defLp}
H^\infty_r(G,p) = \big\{ f \in \ell^\infty(G) \, : \, f * p = f \big\}
\qand
\cL_p = \big\{ \eta \in \cM \, : \, p * \eta = \eta \big\}.
\end{equation}
Note that the constant functions on $G$ are always contained in  $H^\infty_r(G,p)$, and a straightforward 
application of Kakutani's fixed point theorem shows that the set $\cL_p$ is always non-empty. We shall refer to the  
elements in $H^\infty_r(G,p)$ as \textbf{$p$-harmonic functions}, and to the elements in $\cL_p$ 
as \textbf{$p$-stationary means}. Note that if $G = \bF_r$, then the set $\cS_r$ is contained in $\cL_{\sigma_1}$. \\

We say that $(G,p)$ is \textbf{Liouville} if $H_r^\infty(G,p)$ only consists
of constant functions. It is not hard to show that if $G$ is non-amenable, then $(G,p)$ is never Liouville for
any admissible probability measure $p$ on $G$. The converse does not hold on the nose; however, as was
proved by Rosenblatt \cite{Ro} and Kaimanovich-Vershik \cite{KV} (independently), if $G$ is amenable, then
there is always \emph{at least one} (symmetric) admissible probability measure $p$ on $G$ such that $(G,p)$ is 
Liouville. Note that if $\eta \in \cL_p$ and $\phi \in \ell^\infty(G)$, then $f_\phi(g) = \eta(g^{-1} \cdot \phi)$ belongs to $H^\infty_r(G,p)$.
Hence, if the measured group $(G,p)$ is Liouville (in particular, $G$ is amenable), then $f_\phi$ must be constant for
every $\phi \in \ell^\infty(G)$, and thus $\eta$ belongs to the weak*-closed and convex set $\cL_G$ of \textbf{left-invariant means} on $G$ defined by
\begin{equation}
\label{defLG}
\cL_G = \big\{ \eta \in \cM \, : \, s \cdot \eta = \eta, \enskip \textrm{for all $s \in G$} \big\}.
\end{equation}
We conclude that if $(G,p)$ is Liouville, then $\cL_p = \cL_G$. \\

Let $(G,p)$ be a measured group and let $\cF_p$ denote the set of all $\lambda \in \cL_p$ such that 
\begin{equation}
\label{defFp}
\lambda(f) = f(e), \quad \textrm{for all $f \in H^\infty_r(G,p)$}.
\end{equation}
Such means always exist, and we shall refer to them as \textbf{Furstenberg means} (a variant of these means were first introduced in the paper \cite{BF3} by the author and Fish). To show the existence, let $(\beta_n)$ denote the sequence of means 
on $G$ defined by
\begin{equation}
\label{exFp}
\beta_n(f) = \frac{1}{n} \sum_{k=1}^{n} p^{*k}(f), \quad \textrm{for $f \in \ell^\infty(G)$}. 
\end{equation}
Then each weak*-cluster point $\lambda$ of the sequence $(\beta_n)$ belongs to $\cL_p$, and for every 
$f \in H^\infty_r(G,p)$, we have $\beta_n(f) = f(e)$, and thus $\lambda(f) = f(e)$. In particular, $\lambda \in \cF_p$. We note that in the case when $(G,p) = (\bF_r,\sigma_1)$, then the Hecke recurrence relations imply that the inclusion $\cS_r \subset \cF_{\sigma_1}$ holds.

\subsection{Product set phenomena}

Let $G$ be a countable group. Given a set $\cC$ of means on $G$, we define the \textbf{upper} and 
\textbf{lower $\cC$-density} of a set $B \subset G$ by
\[
d^*_{\cC}(B) = \sup\big\{ \lambda(B) \, : \, \lambda \in \cC \big\}
\qand
d_*^{\cC}(B) = \inf\big\{ \lambda(B) \, : \, \lambda \in \cC \big\}
\]
respectively. We say that $B \subset G$ is \textbf{$\cC$-large} if $d^*_{\cC}(B) > 0$. This notion of largeness
encompasses the two notions introduced earlier: 
\begin{itemize}
\setlength\itemsep{1em}
\item If $G$ is amenable, $(F_n)$ is a F\o lner sequence in $G$ and $(\beta_n)$ is defined as in \eqref{defbetan},
then 
\[
d^{*}_{\cF}(B) = \beta^*(B) \qand d_*^{\cF}(B) = \beta_*(B), \quad \textrm{for all $B \subset G$},
\]
where $\cF \subset \cL_G$ denotes the set of all weak*-cluster points of $(\beta_n)$. In particular, a set $B \subset G$
is large with respect to $(F_n)$ if and only if it is $\cF$-large. Furthermore, it is not hard to show (see e.g. 
Lemma 3.3 in \cite{BBF}) that a set is large (with respect to \emph{some} F\o lner sequence in $G$) if and only if it is $\cL_G$-large, where $\cL_G$ is defined as in \eqref{defLG}. 

\item If $G = \bF_r$ and $\cS_r \subset \cM$ is defined as the set of weak*-cluster points of the sequence $(\beta_n)$
in \eqref{sigmaseq}, then 
\[
d^*_{\cS_r}(B) = \overline{s}_r(B) \qand d_*^{\cS_r}(B) = \underline{s}_r(B), \quad \textrm{for all $B \subset \bF_r$}.
\]
In particular, a set $B \subset \bF_r$ is spherically large if and only if it is $\cS_r$-large. Furthermore, since each $\beta_n$ is a symmetric probability measure on $\bF_r$, we also have that $B \subset \bF_r$ is spherically large
if and only if $B^{-1}$ is spherically large.
\end{itemize}

\subsubsection*{Main combinatorial results for amenable groups}

The following theorem, applied to $\cC = \cF$, immediately yields the main results of Di Nasso, Goldbring, Jin, Leth, Lupini and Mahlburg in \cite{BLAJ0} and \cite{BLAJ1}, summarized in \eqref{di1} and \eqref{di2} above.

\begin{theorem}
\label{denamen}
Let $G$ be a countable \textsc{amenable} group and suppose that $A \subset G$ is $\cL_G$-large. For every weak*-closed 
set $\cC \subset \cL_G$, for every $B \subset G$ and $\eps > 0$, there exist \emph{finite} sets $F, K \subset G$ such that
\[
d^*_{\cC}\Big( \bigcap_{l \in L} l FAB \Big) \geq d^*_{\cC}(B),
\qand 
d_*^{\cC}\Big( \bigcap_{l \in L} l KAB \Big) \geq d_*^{\cC}(B) - \eps,
\]
for every finite set $L \subset G$.
\end{theorem}

\begin{remark}
The lower bound for $d_*^{\cF}$, where $\cF$ is defined as above for a fixed choice of F\o lner sequence in $G$, 
was first proved in the case when $G = \bZ^d$ in \cite{BLAJ0}. During the work-shop "Ergodic theory meets Combinatorics" at Banff in July 2015, Goldbring stated the general version (for amenable groups) as an open problem, and shortly thereafter the author of this paper communicated a solution to Goldbring. The six authors of \cite{BLAJ0} and \cite{BLAJ1} were later able to adapt their techniques to prove this version (for $\cC = \cF)$. 
\end{remark}

\subsubsection*{Main combinatorial results for measured groups}

The following theorem, applied to the measured group $(\bF_r,\sigma_1)$, the weak*-closed subset $\cC = \cS_r \subset \cL_{\sigma_1}$ and a spherically large set $A \subset G$, immediately yields Theorem \ref{spherical}. Indeed, we note that if $A$ is spherically large, then so is $A^{-1}$ (by the symmetry of $\sigma_k$), and thus $\cF_{\sigma_1}$-large by the inclusion $\cS_r \subset \cF_{\sigma_1}$.

\begin{theorem}
\label{dengen}
Let $(G,p)$ be a measured group and suppose that $A \subset G$ is $\cF_p$-large. For every weak*-closed set $\cC \subset \cL_p$, for every $B \subset G$ and $\eps > 0$, there exist \emph{finite} sets $F, K \subset G$ such that
\[
d^*_{\cC}\Big( \bigcap_{l \in L} l A^{-1}FB \Big) \geq d^*_{\cC}(B)
\qand 
d_*^{\cC}\Big( \bigcap_{l \in L} l A^{-1}KB \Big) \geq d_*^{\cC}(B) - \eps,
\]
for every finite set $L \subset G$.
\end{theorem}

\begin{remark}
In the recent paper \cite{BF3}, it is claimed (Theorem 1.1) that if $(G,p)$ is a (symmetric) measured group, $A$ is $\cF_p$-large and $B$ is $\cL_p$-large, then $AB$ is piecewise syndetic. However, in the proof of this theorem, two (connected) serious left/right mistakes are made. Lemma \ref{Fpext} below corrects one of these 
mistakes (Corollary 3.15 in \cite{BF3}). More seriously, in a crucial passage (in the proof of Lemma 5.5 in \cite{BF3}), the defining property \eqref{Fpext} for a mean $\lambda \in \cF_p$ is used for a \emph{left} $p$-harmonic function (fixed point for convolution with $p$ on the left). Such a function is never $p$-harmonic unless it is constant. This is not a problem for elements in $\cF_p$ which are weak* cluster points of the sequence \eqref{exFp}, but at the place in \cite{BF3} where Lemma 5.5 is used (in the proof of Proposition 5.4), no freedom in choosing $\lambda \in \cF_p$ is allowed. 
During an attempt to rectify these mistakes, it was realized that the circle of ideas presented in \cite{BF3} can only be used to prove that there exists a finite set $F \subset G$ such that $A^{-1}FB$ is thick. Theorem \ref{dengen} can be viewed as a "quantification" of this correction. 
\end{remark}

We stress that if $\cF_p$ is replaced with $\cL_p$, then the finite set $F$ in Theorem \ref{dengen} cannot be moved to the left of $A^{-1}$; in fact, $A^{-1}B$ need not even be piecewise syndetic in general as the following result shows (the proof is given in Section \ref{sec:prfthmnotliou}).

\begin{theorem}
\label{notLiouville}
There is a measured group $(G,p)$ and subsets $A, B \subset G$ such that
\begin{itemize}
\item $A$ is thick,
\item $B$ is $\cF_p$-large and left thick, and
\item $A^{-1}B$ is \textsc{not} piecewise syndetic. 
\end{itemize}
In fact, $G$ can be chosen to be amenable.
\end{theorem}

\begin{remark}
It is not unlikely that these kinds of examples can be constructed for \emph{every} non-Liouville measured group; 
at least we do not know of an example of a non-Liouville measured group for which the construction in Section 
\ref{sec:prfthmnotliou} does not work. Note that if $(G,p)$ is Liouville (and thus $\cF_p = \cL_p = \cL_G$), $A$ is
thick and $B$ is $\cL_G$-large, then $A^{-1}B$ is always thick (so in particular, piecewise syndetic).
\end{remark}

The following positive result is established in Section \ref{sec:prfpwFol} and generalizes Theorem \ref{spherical2}
above.

\begin{theorem}
\label{piecewiseFol}
Suppose that $A \subset G$ is $\cL_p$-large. Then $AA^{-1}$ is both piecewise syndetic and piecewise left syndetic. 
\end{theorem}

In the appendix of this paper, we give examples of
measured groups $(G,p)$ and $\cF_p$-large subsets thereof whose difference sets are \emph{not} syndetic. Such examples were first given by the author
and Fish in \cite{BF3} for free groups, but here we show that these examples can be further generalized to even encompass 
certain \emph{amenable} measured groups. Since difference sets of $\cL_G$-large sets in amenable groups are
always syndetic by a classical theorem of F\o lner \cite{Fol}, amenable examples of such measured groups must necessarily be non-Liouville.

\subsection{Product set phenomena with respect to a fixed mean}

We are now ready to formulate our two main (combinatorial) results, and deduce Theorem 
\ref{denamen} and Theorem \ref{dengen} from them. 

\subsubsection{Amenable groups}

Let us briefly recall an an important notational convention in this paper: If $\lambda$ is a mean on $G$ and $A \subset G$ is a subset, then we abuse notation 
and write $\lambda(A)$ to denote $\lambda(\chi_A)$, where $\chi_A$ is the indicator function for the 
set $A$. 

\begin{theorem}
\label{mainamen}
Let $G$ be a countable \textsc{amenable} group and suppose that $A \subset G$ is a $\cL_G$-large set. Fix $\lambda \in \cL_G$. For every $B \subset G$, there exists a \emph{finite} set $F \subset G$ such that
\[
\lambda\Big( \bigcap_{l \in L} l FAB \Big) \geq \lambda(B), \quad \textrm{for every finite set $L \subset G$}.
\]
Furthermore, for every $\eps > 0$, one can find a finite set $K \subset G$, which only depends on
the set $A$, such that for \emph{every} $B \subset G$,
\begin{equation}
\label{lowbndamen}
\lambda\Big( \bigcap_{l \in L} l KAB \Big) \geq \lambda(B) - \eps, \quad \textrm{for every finite set $L \subset G$}.
\end{equation}
\end{theorem}

\begin{remark}
Since the set $K$ does \emph{not} depend on $\lambda$ or $B$, the lower bound 
\eqref{lowbndamen} holds \emph{uniformly} over all $\lambda$, $B$ and finite sets $L$.
\end{remark}

\begin{proof}[Proof of Theorem \ref{denamen} assuming Theorem \ref{mainamen}]
Fix a weak*-closed set $\cC \subset \cL_G$ and $B \subset G$. By continuity of the 
map $\lambda \mapsto \lambda(B)$ on $\cC$, we can find $\lambda \in \cC$ 
such that $d^*_{\cC}(B) = \lambda(B)$. By Theorem \ref{mainamen}, we know that 
there exists a finite set $F \subset G$ such that for every finite set $L \subset G$, we
have
\[
d^*_{\cC}\Big( \bigcap_{l \in L} l FAB \Big) \geq \lambda\Big( \bigcap_{l \in L} l FAB \Big) \geq \lambda(B) = d^*_{\cC}(B),
\]
which proves the first assertion in Theorem \ref{denamen}. To prove the second assertion, we again
use Theorem \ref{mainamen} to produce, for every $\eps > 0$, a finite set $K \subset G$ such that
\begin{equation}
\label{lowbnd}
\lambda\Big( \bigcap_{l \in L} l KAB \Big) \geq \lambda(B) - \eps, \quad \textrm{for every $\lambda \in \cL_G$ and finite set $L \subset G$}.
\end{equation}
For every finite set $L \subset G$, there exists $\lambda_L \in \cC$ such that
\[
d_*^{\cC}
\Big(
\bigcap_{l \in L} l KAB 
\Big) 
= 
\lambda_L
\Big(
\bigcap_{l \in L} l KAB 
\Big),
\]
and thus, by \eqref{lowbnd}
\[
d_*^{\cC}
\Big(
\bigcap_{l \in L} l KAB 
\Big) 
= 
\lambda_L
\Big(
\bigcap_{l \in L} l KAB 
\Big)
\geq 
\lambda_L(B) - \eps \geq d_*^{\cC}(B) - \eps,
\]
which finishes the proof of the second assertion in Theorem \ref{denamen}.
\end{proof}

\subsubsection{Measured groups}

\begin{theorem}
\label{maingen}
Let $(G,p)$ be a countable \textsc{measured group} and suppose that $A \subset G$ is $\cF_p$-large. Fix $\lambda \in \cL_p$. For every $B \subset G$, there exists a \emph{finite} set $F \subset G$ such that
\[
\lambda\Big( \bigcap_{l \in L} l A^{-1}FB \Big) \geq \lambda(B).
\]
Furthermore, for every $\eps > 0$, one can find a finite set $K \subset G$, which only depends on
the set $A$, such that 
\[
\lambda\Big( \bigcap_{l \in L} l A^{-1}KB \Big) \geq \lambda(B) - \eps.
\]
\end{theorem}

One can prove Theorem \ref{dengen} from Theorem \ref{maingen} along the same lines as in the proof of 
Theorem \ref{denamen} above, assuming Theorem \ref{mainamen}. We leave the details to the reader.

\subsection{Action set phenomena}
\label{secact}
In Section \ref{prfmain} below we show how the main results in this subsection can be used to prove
Theorem \ref{mainamen} and Theorem \ref{maingen}. We shall use the same notation throughout the
paper. Let $(G,p)$ be a countable 
measured group. 
Suppose that $\overline{Y}$ is a compact and second countable space equipped with a homeomorphic 
action of $G$. Let $C(\overline{Y})$ denote the space of real-valued continuous functions on $\overline{Y}$
and note that $G$ acts on $C(\overline{Y})$ by $(s \cdot f)(y) = f(s^{-1} \cdot y)$. A Borel probability measure 
$\nu$ on $\overline{Y}$ is called \textbf{$p$-stationary} if
\[
\sum_{s} p(s) \nu(s^{-1} \cdot f) = \nu(f), \quad \textrm{for all $f \in C(\overline{Y})$},
\]
and \textbf{$G$-invariant} if $\nu(s^{-1} \cdot f) = \nu(f)$ for all $s \in G$ and $f \in C(\overline{Y})$. Clearly,
every $G$-invariant Borel probability measure is $p$-stationary, but the converse does not hold unless $(G,p)$
is Liouville. If $Y \subset \overline{Y}$ is a $G$-invariant Borel set with $\nu(Y) = 1$, then we shall view $\nu$
as a Borel probability measure on $Y$, and refer to $(Y,\nu)$ as a \textbf{$(G,p)$-space}. If $\nu$ is $G$-invariant,
we say that $(Y,\nu)$ is a \textbf{probability measure preserving} (or p.m.p.\ for short) \textbf{$G$-space}.
In both cases, we say that $(Y,\nu)$ is \textbf{ergodic} if a $G$-invariant Borel subset of $Y$ is either $\nu$-null 
or $\nu$-conull. Finally, if $A \subset G$ and $B \subset Y$ is a Borel set, we define the \textbf{action set} $AB$ as
\[
AB = \bigcup_{a \in A} aB = \big\{ a \cdot b \, : \, a \in A, \: b \in B \big\} \subset Y,
\]
which is again a Borel subset of $Y$. 

\subsubsection{Main result for p.m.p.\ $G$-spaces}

The following result is established in Section \ref{sec:prfsmainpmpandmainharm}.

\begin{theorem}
\label{mainpmp}
Let $(Y,\nu)$ be a p.m.p.\ $G$-space and let $A \subset G$ be $\cL_p$-large. For every Borel set 
$B \subset Y$, there exists a \emph{finite} set $F \subset G$ such that
\[
\nu\Big( \bigcap_{g \in G} g FAB \Big) \geq \nu(B).
\]
Furthermore, for every $\eps > 0$, one can find a finite set $K \subset G$, which only depends on
the set $A$, such that 
\[
\nu\Big( \bigcap_{g \in G} g KAB \Big) \geq \nu(B) - \eps.
\]
\end{theorem}

\begin{remark}
We stress that we do allow $G$ to be non-amenable. 
\end{remark}

\subsubsection{Main result for $(G,p)$-spaces}

The following result is established in Section \ref{sec:prfsmainpmpandmainharm}.

\begin{theorem}
\label{mainharm}
Let $(Y,\nu)$ be a $(G,p)$-space and let $A \subset G$ be $\cF_p$-large. For every Borel set
$B \subset Y$, there exists a \emph{finite} set $F \subset G$ such that
\[
\nu\Big( \bigcap_{g \in G} g A^{-1}FB \Big) \geq \nu(B).
\]
Furthermore, for every $\eps > 0$, one can find a finite set $K \subset G$, which only depends on
the set $A$, such that 
\[
\nu\Big( \bigcap_{g \in G} g A^{-1}KB \Big) \geq \nu(B) - \eps. 
\]
\end{theorem}

\begin{remark}
We stress that in both theorems above, the set $K$ does \emph{not} depend on the $G$-space $(Y,\nu)$
at hand.
\end{remark}

\section{Proofs of Theorem \ref{mainamen} and Theorem \ref{maingen}}
\label{prfmain}

Assuming Theorem \ref{mainpmp} and Theorem \ref{mainharm} we shall in this section prove Theorem 
\ref{mainamen} and Theorem \ref{maingen}. For the proofs, as well as for other purposes later in the text,
we shall need the notion of a \emph{Bebutov triple} associated to a subset of a countable group.

\subsection{Bebutov triples}
Let $G$ be a countable group and let $2^G$ denote the space of all subsets of $G$ endowed with the 
Tychonoff topology, which makes it into a compact and second countable space. It comes equipped with the
homeomorphic $G$-action given by $g \cdot A = Ag^{-1}$ for all $g \in G$ and $A \subset G$. Define 
the \emph{clopen} set $U \subset 2^G$ by
\[
U = \big\{ A \in 2^G \, : \, e \in A \big\},
\]
where $e$ denotes the identity element in $G$. We note that
\[
U_A := \big\{ g \in G \, : \, g \cdot A \in U \big\} = A, \quad \textrm{for all $A \subset G$}.
\]
Fix a subset $A \subset G$ and let $X$ denote the closure in $2^G$ of the $G$-orbit of $A$. Then
$X$ is again a compact and second countable space equipped with a homeomorphic $G$-action 
and, by construction, $x_o := A$ has a dense $G$-orbit in $X$. We abuse notation and define 
the clopen subset $A := U \cap X \subset X$, so that
\[
A_{x_o} = \big\{ g \in G \, : \, g \cdot x_o \in A \big\} = A,
\]
where the symbol $A$ within the parenthesis is viewed as a clopen subset of $X$, while the symbol $A$ 
on the right hand side is viewed as a subset of $G$. We refer to $(X,x_o,A)$ as the \textbf{Bebutov triple of $A$}, 
and once it has been introduced for a given set $A \subset G$, we shall write $A_{x_o}$ to denote this set, and
we write $A$ for the clopen set in $X$.

\subsection{Compact pointed $G$-spaces}

It will be useful to adopt a more general point of view on Bebutov triples. Suppose that $X$ is a compact
and second countable $G$-space equipped with a homeomorphic action of $G$, and suppose that $x_o \in X$
has a dense $G$-orbit. We shall refer to $(X,x_o)$ as a \textbf{compact pointed $G$-space}. Given a subset $A \subset X$ 
and $x \in X$, we write
\[
A_x = \big\{ g \in G \, : \, g \cdot x \in A \big\}.
\]
Define the positive and unital linear map $T : C(X) \ra \ell^\infty(G)$ by
\begin{equation}
\label{defBebutov}
(T\phi)(g) = \phi(g \cdot x_o), \quad \textrm{for $\phi \in C(X)$}.
\end{equation}
We shall refer to $T$ as the \textbf{Bebutov map} of $(X,x_o)$. Note that $G$ acts on both $C(X)$ and $\ell^\infty(G)$ by
\[
(s \cdot \phi)(x) = \phi(s^{-1} \cdot x) \qand (s \cdot f)(g) = f(s^{-1}g), \quad \textrm{for $\phi \in C(X)$ and $f \in \ell^\infty(G)$},
\]
and one readily checks that $T$ is $G$-equivariant, i.e. $s \cdot T\phi = T(s \cdot \phi)$ for all $s \in G$. Let 
$\cP(X)$ denote the space of Borel probability measures on $X$, which is a weak*-compact and convex subset
of the dual space $C(X)^*$, and is equipped with the affine and weak*-homeomorphic $G$-action given by
\[
(s \cdot \mu)(\phi) = \mu(s^{-1} \cdot \phi), \quad \textrm{for $\phi \in C(X)$}.
\]
Since $T$ is positive and unital, we have $T^*\cM \subset \cP(X)$, where $\cM$ denotes the set of means on $G$. 
Furthermore, if $G$ is \emph{amenable}, then we have $T^*\cL_G \subset \cP_G(X)$, where
\[
\cP_G(X) = \big\{ \mu \in \cP(X) \, : \, s \cdot \mu = \mu, \enskip \textrm{for all $s \in G$} \big\}.
\]
We stress that $\cP_G(X)$ might be empty if $G$ is not amenable. More generally, given an admissible probability
measure $p$ on $G$, we define
\[
\cP_p(X) = \big\{ \mu \in \cP(X) \, : \,\sum_{s \in G} p(s) (s^{-1} \cdot \mu) = \mu, \enskip \textrm{for all $s \in G$} \big\}.
\]
A straightforward application of Kakutani's fixed point theorem shows that $\cP_p(X)$ is always non-empty. We refer
to the elements in $\cP_p(X)$ as \textbf{$p$-stationary} (or \textbf{$p$-harmonic}) Borel probability measures on $X$.
Since $T$ is $G$-equivariant, we have $T^*\cL_p \subset \cP_p(X)$, where $\cL_p \subset \cM$ denotes the set of 
$p$-stationary means on $G$ defined in \eqref{defLp}. \\

Suppose that $\eta \in \cM$ and define $\nu = T^*\eta \in \cP(X)$. Then, for every \emph{clopen} set 
$U \subset X$, its indicator function $\chi_U$ belong to $C(X)$ and thus
\[
\eta(U_{x_o}) = \eta(T\chi_U) = \nu(\chi_U) = \nu(U).
\]
If $U$ is only assumed to be open, this identity turns into a lower bound, as the following lemma shows.
\begin{lemma}
\label{open}
For every open set $U \subset X$ and $\eta \in \cM$, we have $\eta(U_{x_o}) \geq \nu(U)$, where 
$\nu = T^*\eta$.
\end{lemma}

\begin{proof}
Fix $\eta \in \cM$ and an open set $U \subset X$. Since $X$ is second countable, we can find an increasing
sequence $(f_n)$ in $C(X)$ such that $\chi_U = \sup_n f_n$. Hence, with $\nu = T^*\eta \in \cP(X)$, we have
\[
\eta(U_{x_o}) = \eta(\sup_n Tf_n) \geq \eta(Tf_m) = \nu(f_m), \quad \textrm{for every $m$}.
\]
By monotone convergence (note that $\nu$ is $\sigma$-additive), this inequality is preserved upon taking the limit $m \ra \infty$.
\end{proof}

Note that if $A, L \subset G$ and $B \subset X$, then
\[
\Big(\bigcap_{l \in L} l AB\Big)_x = \bigcap_{l \in L} l AB_x, \quad \textrm{for all $x \in X$}.
\]
Hence, by specializing the lemma above to sets which are finite intersections of countable unions of 
clopen sets, we conclude:

\begin{corollary}
\label{corlowbnd}
For every $A \subset G$, clopen set $B \subset X$ and $\eta \in \cM$, we have
\[
\eta(B_{x_o}) = \nu(B) \qand \eta\Big( \bigcap_{l \in L} l AB_{x_o} \Big) \geq \nu\Big( \bigcap_{l \in L} l AB \Big), \quad \textrm{for every \emph{finite} set $L \subset G$},
\]
where $\nu = T^*\eta$.
\end{corollary}

\subsection{Proof of Theorem \ref{mainamen} assuming Theorem \ref{mainpmp}}

Let $G$ be a countable amenable group and suppose that $A \subset G$ is $\cL_G$-large. We fix $\lambda \in \cL_G$
and a subset $B \subset G$. We abuse notation and write $(Y,y_o,B)$ for the Bebutov triple of the set $B$ in $G$,
which we henceforth denote by $B_{y_o}$. Let $T : C(Y) \ra \ell^\infty(G)$ be the linear map $T\phi(g) = \phi(g \cdot y_o)$. From our discussion above, we know that $\nu = T^*\lambda$ belongs to $\cP_G(Y)$, and by Theorem 
\ref{mainpmp} applied to the p.m.p.\ $G$-space $(Y,\nu)$, we can find a finite set $F \subset G$ such that 
\[
\nu\Big(\bigcap_{l \in L} l FAB\Big) \geq \nu(B), \quad \textrm{for every finite set $L \subset G$}.
\]
Since $B \subset Y$ is clopen, we conclude from Corollary \ref{corlowbnd} that
\[
\eta(B_{y_o}) = \nu(B) \qand \eta\Big(\bigcap_{l \in L} l FAB_{y_o}\Big) \geq \nu\Big(\bigcap_{l \in L} l FAB\Big),
\]
and thus
\[
\eta\Big(\bigcap_{l \in L} l FAB_{y_o}\Big) \geq \nu(B) = \eta(B_{y_o}),
\]
for every finite set $L \subset G$, which finishes the proof of first assertion in Theorem \ref{mainamen}. For the
second assertion, we note that Theorem \ref{mainpmp} also provides us with, for every $\eps > 0$, a finite set 
$K \subset G$, which only depends on the set $A$, but not on $B$ or the p.m.p.\ $G$-space at hand, such that
\[
\nu\Big(\bigcap_{l \in L} l KAB\Big) \geq \nu(B) - \eps, \quad \textrm{for every finite set $L \subset G$}.
\]
The same argument as above now yields:
\[
\eta\Big(\bigcap_{l \in L} l KAB_{y_o}\Big) \geq \nu(B) - \eps = \eta(B_{y_o}) - \eps,
\]
which finishes the proof of the second assertion.

\subsection{Proof of Theorem \ref{maingen} assuming Theorem \ref{mainharm}}

The proof of Theorem \ref{maingen} works almost verbatim the same as in the previous 
subsection (note that if $\lambda \in \cL_p$, then $T^*\lambda \in \cP_p(Y)$ and thus $(Y,\nu)$
is a $(G,p)$-space), with the important difference that one has to consider product/action sets of the 
form $A^{-1}FB$ and $A^{-1}KB$ instead of $FAB$ and $KAB$. In any case, Corollary \ref{corlowbnd}
works also in this case. We leave the details to the reader.


\section{Proofs of Theorem \ref{mainpmp} and Theorem \ref{mainharm}}
\label{sec:prfsmainpmpandmainharm}
The proofs of Theorem \ref{mainpmp} and Theorem \ref{mainharm} are very similar in spirit and will 
run in parallel. We begin by breaking down the proofs into three separate steps. The third step is a
rather minor one and will be discussed here, while the two first steps are more involved, and details are 
postponed to later sections.

\subsection{Step I: Expansion of action sets in ergodic $(G,p)$-spaces}

Let $(G,p)$ be a countable measured group. The definitions of p.m.p.\ $G$-spaces and $(G,p)$-spaces are given 
in Subsection \ref{secact} above. 

\begin{lemma}
\label{sumexceed_pmp}
Suppose that $A \subset G$ is $\cL_p$-large. For every \emph{ergodic} p.m.p.\ $G$-space $(Y,\nu)$ and Borel 
set $B \subset Y$ such that $d^*_{\cL_p}(A) + \nu(B) > 1$, we have $\nu(AB) = 1$.
\end{lemma}

\begin{lemma}
\label{sumexceed_harm}
Suppose that $A \subset G$ is $\cF_p$-large. For every \emph{ergodic} $(G,p)$-space $(Y,\nu)$ and  
Borel set $B \subset Y$ such that $d^*_{\cF_p}(A) + \nu(B) > 1$, we have $\nu(A^{-1}B) = 1$.
\end{lemma}

\begin{remark}
As we shall see in Section \ref{sec:prfthmnotliou}, Lemma \ref{sumexceed_harm} fails miserably if the set 
$\cF_p$ of Furstenberg means is replaced with the set $\cL_p$ of \emph{all} $p$-stationary means 
(for non-Liouville measured groups). 
\end{remark}

\subsection{Step II: Ergodicity for $p$-stationary densities}

Recall the definitions of the sets $\cL_p$ and $\cF_p$ 
from \eqref{defLp} and \eqref{defFp} respectively. 

\begin{lemma}
\label{lemmaleftergLp}
Let $A \subset G$ be a $\cL_p$-large set. For every $\eps > 0$, there exists a finite set $F \subset G$ such that $d^*_{\cL_p}(FA) > 1 - \eps$.
\end{lemma}

\begin{lemma}
\label{lefterg}
Let $A \subset G$ be a $\cF_p$-large set. For every $\eps > 0$, there exists a finite set $F \subset G$ such that $d^*_{\cF_p}(FA) > 1 - \eps$.
\end{lemma}

Even though the two lemmas look quite  similar, the second one is significantly more involved to prove (the proof is
given in Subsection \ref{subsecerg} below). We stress that unless $(G,p)$ is Liouville, neither of the two lemmas 
above imply the other. 

\subsection{Step III: Ergodic decomposition}

Suppose that $(Y,\nu)$ is either a p.m.p.\ $G$-space or a $(G,p)$-space. We stress that we do not assume
that $\nu$ is ergodic. Let $\cP_G^{\textrm{erg}}(X)$ and $\cP_p^{\textrm{erg}}(X)$ denote the set of \emph{ergodic}
measures in $\cP_G(X)$ and $\cP_p(X)$ respectively. By either Theorem 4.8 in \cite{EW} or Corollary 2.7 in 
\cite{BS}, $\nu$ admits an \emph{ergodic decomposition}, i.e. a Borel probability measure $\kappa$ on either $\cP_G^{\textrm{erg}}(X)$ or $\cP^{\textrm{erg}}_p(X)$ respectively , such that
\[
\nu(B) = \int_{\cP(Y)} \mu(B) \, d\kappa(\mu), \quad \textrm{for every Borel set $B \subset Y$}.
\]
Given a Borel set $B \subset Y$ and $\eps > 0$, we define
\[
B_\eps = \big\{ \mu \in \cP_{\bullet}^{\textrm{erg}}(Y) \, : \, \mu(B) \geq \eps \big\}, \quad \textrm{where $\bullet=G$ 
or $p$}.
\]
and a straightforward application of Fubini's Theorem shows that
\[
\int_0^1 \kappa(B_\eps) \, d\eps = \nu(B), \quad \textrm{for every Borel set $B \subset Y$}.
\]
The following two lemmas are now immediate. 
\begin{lemma}
\label{eps1}
For every Borel set $B \subset Y$, there exists $\eps > 0$ such that $\kappa(B_\eps) \geq \nu(B)$.
\end{lemma}

\begin{lemma}
\label{eps2}
For every Borel set $B \subset Y$ and $\eps > 0$, we have $\kappa(B_\eps) \geq \nu(B) - \eps$.
\end{lemma}

\begin{proof}
Note that for every $\eps > 0$,
\[
\nu(B) = \int_{B_\eps^c} \mu(B) \, d\kappa(\mu) + \int_{B_\eps} \mu(B) \, d\kappa(\mu) \leq \eps + \kappa(B_\eps),
\]
and thus $\kappa(B_\eps) \geq \nu(B) - \eps$.
\end{proof}

\subsection{Proof of Theorem \ref{mainpmp}}

Let $(G,p)$ be a countable measured group and suppose that $(Y,\nu)$ is a p.m.p.\ $G$-space. Let $A \subset G$ be a 
$\cL_p$-large set, and let $B \subset Y$ be a Borel set with positive $\nu$-measure. Let $\kappa$ denote the
ergodic decomposition of $\nu$. By Lemma \ref{eps1} we can find $\eps > 0$ such that $\kappa(B_\eps) \geq \nu(B)$,
and by Lemma \ref{lemmaleftergLp} we can find a finite subset $F \subset G$ such that $d^*_{\cL_p}(FA) > 1 - \eps$.
Since $B_\eps$ consists of ergodic measures, and $d^*_{\cL_p}(FA) + \mu(B) > 1$ for all $\mu \in B_\eps$, 
Lemma \ref{sumexceed_pmp} guarantees that $\mu(FAB_\eps) = 1$ for all $\mu \in B_\eps$. Hence,
\[
\mu\Big(\bigcap_{g \in G} gFAB\Big)  = 1, \quad \textrm{for all $\mu \in B_\eps$},
\]
and thus
\[
\nu\Big(\bigcap_{g \in G} gFAB\Big) \geq \int_{B_\eps} \mu\Big(\bigcap_{g \in G} gFAB\Big) \, d\kappa(\mu) = \kappa(B_\eps) \geq \nu(B),
\]
which finishes the proof of the first assertion in Theorem \ref{mainpmp}. For the second assertion, let us fix $\eps > 0$, 
and choose, by Lemma \ref{lemmaleftergLp} a finite set $K \subset G$ such that $d^*_{\cL_p}(KA) > 1 - \eps$. 
By Lemma \ref{eps2}, we have $\kappa(B_\eps) \geq \nu(B) - \eps$, and by Lemma \ref{sumexceed_pmp}, we
have
\[
\mu\Big(\bigcap_{g \in G} gKAB\Big)  = 1, \quad \textrm{for all $\mu \in B_\eps$},
\]
and thus
\[
\nu\Big(\bigcap_{g \in G} gKAB\Big) \geq \int_{B_\eps} \mu\Big(\bigcap_{g \in G} gKAB\Big) \, d\kappa(\mu) = \kappa(B_\eps) \geq \nu(B) - \eps,
\]
which finishes the proof of the second assertion in Theorem \ref{mainpmp}.

\subsection{Proof of Theorem \ref{mainharm}}

The proof works almost verbatim the same as in the previous subsection with the following minor modifications: Instead of 
Lemma \ref{lemmaleftergLp}, we use Lemma \ref{lefterg}, and instead of Lemma \ref{sumexceed_pmp}, we use
Lemma \ref{sumexceed_harm}.

\section{Proofs of Lemma \ref{sumexceed_pmp} and Lemma \ref{sumexceed_harm}}

\subsection{Ergodic theorems}

Let us briefly recall our conventions from Subsection \ref{secact}. We say that $(Y,\nu)$ is a $(G,p)$-space
if there exists a compact and second countable space $\overline{Y}$, equipped with an action of $G$ by 
homeomorphisms, such that $Y$ is a $G$-invariant Borel set of $\overline{Y}$ and $\nu$ is a Borel probability
measure on $\overline{Y}$ such that $\nu(Y) = 1$. Finally, if $\eta$ is a mean on $G$ and $\phi$ is a bounded real-valued function on $G$, we write 
\[
\eta(\phi) = \int_G \phi(g) \, d\eta(g), 
\]
even though the right-hand side is not an integral in the Lebesgue sense. 

\begin{lemma}[Weak Ergodic Theorem]
\label{harmergthm}
Let $(Y,\nu)$ be an ergodic $(G,p)$-space and let $\eta \in \cF_p$. For all Borel sets $B, C \subset Y$, 
we have
\begin{equation}
\label{harmwet}
\int_G \nu(g^{-1}B \cap C) \, d\eta(g) = \nu(B) \, \nu(C).
\end{equation}
\end{lemma}

\begin{proof}
Given a Borel set $C \subset Y$ with positive measure and $\phi \in C(\overline{Y})$, we define
\[
f_\phi(g) = \frac{\nu((g^{-1} \cdot \phi)\chi_C)}{\nu(C)}, \quad \textrm{for $g \in G$}.
\]
Given any $\eta \in \cM$, one readily checks that the functional $\nu_C(\phi) := \eta(f_\phi)$ on $C(\overline{Y})$,
is unital and positive, and thus a Borel probability measure on 
$\overline{Y}$. We shall first prove that if $\eta \in \cL_p$, then $\nu_C$ is $p$-stationary. To do this, first note
that
\[
f_{s^{-1} \cdot \phi}(g) = f_\phi(sg) = (s^{-1} \cdot f_\phi)(g), \quad \textrm{for all $g,s \in G$}.
\]
Hence, since $\eta$ is assumed to be left $p$-stationary, 
\begin{eqnarray*}
\sum_{s \in G} p(s) (s \cdot \nu_C)(\phi) 
&=& 
\sum_{s \in G} p(s) \eta(f_{s^{-1} \cdot \phi}) 
=
\sum_{s \in G} p(s) \eta(s^{-1} \cdot f_\phi) \\
&=&
\sum_{s \in G} p(s) (s \cdot \eta)(f_\phi) 
= 
\eta(f_\phi) = \nu_C(\phi),
\end{eqnarray*}
for every $\phi \in C(\overline{Y})$, which shows that $\nu_C$ is $p$-stationary. We conclude that
\begin{equation}
\label{nuCY}
\nu_Y = \nu(C) \nu_C + \nu(C^c) \nu_{C^c},
\end{equation}
for every Borel set $C \subset Y$ with positive measure, where
\[
\nu_Y(\phi) = \int_G \nu(g^{-1} \cdot \phi) \, d\eta(g), \quad \textrm{for $\phi \in C(\overline{Y})$}.
\]
Let us from now on assume that $\eta \in \cF_p$. We note that since $\nu$ is $p$-stationary, the function 
$k_\phi(g) = \nu(g^{-1} \cdot \phi)$ on $G$ belongs to $H^\infty_r(G,p)$, and thus 
\[
\nu_Y(\phi) = \eta(k_\phi) = k_\phi(e) = \nu(\phi), \quad \textrm{for all $\phi \in C(\overline{Y})$}.
\]
Hence $\nu_Y = \nu$. Since we have assumed that $\nu$ is ergodic, and thus an extreme point in $\cP_p(\overline{Y})$ (see e.g. Corollary 2.7 in \cite{BS}), the identity \eqref{nuCY} now implies that
$\nu = \nu_C = \nu_{C^c}$. In particular,
\begin{equation}
\label{phicont}
\nu(\phi) \nu(C) = \int_G \nu((g^{-1} \cdot \phi) \chi_C) \, d\eta(g), \quad \textrm{for all $\phi \in C(\overline{Y})$}.
\end{equation}
If $\nu$ is $G$-invariant, then a rather straightforward approximation argument shows that one can replace
$\phi$ with any \emph{bounded} Borel measurable function. If $\nu$ is only $p$-stationary, then this approximation argument gets a bit more involved. Assume for now that $\phi$ is a bounded measurable function and fix $\eps > 0$. By Lusin's Theorem, we can find $\psi$ in $C(\overline{Y})$ and a measurable set $E \subset Y$ with $\nu(E) < \eps$ such that $\phi = \psi$ on $\overline{Y} \setminus E$. We note that
\[
\big| \nu((g^{-1} \cdot \phi) \chi_C) - \nu((g^{-1} \cdot \psi) \chi_C) \big| \leq \nu(g^{-1}E)
\]
for all $g \in G$, and thus, since the $\eta$-integral of the term which involves the continuous function $\psi$ equals $\nu(\psi) \nu(C)$
by \eqref{phicont}  above, we have
\[
\big| \int_G \nu((g^{-1} \cdot \phi)\chi_C) \, d\eta(g) - \nu(\phi) \nu(C) \big| \leq \max(\|\phi\|_\infty,\|\psi\|_\infty) \cdot \big(\eps + \int_G \nu(g^{-1}E) \, d\eta(g) \big).
\]
We note that since $\nu$ is $p$-stationary, the function $f(g) = \nu(g^{-1}E)$ is $p$-harmonic. Since $\eta$ belongs
to $\cF_p$ we see that $\eta(f) = f(e) = \nu(E) < \eps$. Since $\eps > 0$ is arbitrary, we conclude that
\begin{equation}
\label{phimeas}
\int_G \nu((g^{-1} \cdot \phi)\chi_C) \, d\eta(g) = \nu(\phi) \nu(C), 
\end{equation}
for every bounded Borel measurable function $\phi$ on $Y$. In particular, we can take $\phi = \chi_B$, where $B$
is a $G$-invariant Borel set. It follows from \eqref{phimeas} that $\nu(B \cap C) = \nu(B)\nu(C)$ for every Borel set 
$C \subset Y$, and thus $\nu(B)$ equals either $0$ or $1$. We conclude that $\nu$ is ergodic. \\
\end{proof}

Let us now list three rather immediate consequences of the proof of the Weak Ergodic Theorem above.

\begin{scholium}
\label{scholiumergthm}
Let $(Y,\nu)$ be an ergodic $(G,p)$-space and let $\eta \in \cF_p$. For all $\phi, \psi \in L^\infty(Y,\nu)$,
we have
\[
\int_G \nu((g^{-1} \cdot \phi) \psi) \, d\eta(g) = \nu(\phi) \, \nu(\psi).
\]
\end{scholium}

\begin{scholium}
\label{pmpergthm}
Let $(Y,\nu)$ be an ergodic p.m.p.\ $G$-space and $\eta \in \cF_p$. For all Borel sets $B, C \subset Y$, we have
\[
\int_G \nu(gB \cap C) \, d\eta(g) = \nu(B) \, \nu(C).
\]
\end{scholium}

\begin{proof}
Since $\nu$ is $G$-invariant, we have $\nu(gB \cap C) = \nu(B \cap g^{-1}C)$ for all $g \in G$, and the 
result follows from the Weak Ergodic Theorem.
\end{proof}

Recall that $\cL_p \subset \cM$ is a weak*-compact and convex subset. We say that $\lambda \in \cL_p$ is 
\textbf{extreme} if it cannot be written as a non-trivial convex combination of distinct elements in $\cL_p$. By
Krein-Milman's Theorem, extreme points always exist. We denote by $\cL_p^{\textrm{ext}}$ the set of extreme
points in $\cL_p$.

\begin{scholium}
\label{harmergthmmean}
Let $\lambda \in \cL_p^{\textrm{ext}}$ and $\eta \in \cF_p$. For all $B, C \subset G$, we have
\[
\int_G \lambda(g^{-1}B \cap C) \, d\eta(g) = \lambda(B) \, \lambda(C).
\]
More generally, for all $f_1, f_2 \in \ell^\infty(G)$, we have
\[
\int_G \lambda((g^{-1} \cdot f_1)f_2) \, d\eta(g) = \lambda(f_1) \lambda(f_2).
\]
\end{scholium}

\begin{proof}
Except for the approximation of Borel functions with continuous functions, everything in the proof of the Weak
Ergodic Theorem above could have been done for finitely additive probability measures on any set equipped 
with a $G$-action; in particular, we could have applied it to $Y = G$ and $\nu = \lambda$, where $\lambda$ is 
an \emph{extremal} ("ergodic") mean in $\cL_p$. Of course, in this case, no Borel approximation argument is required.
\end{proof}

\subsection{Proof of Lemma \ref{sumexceed_harm}}
Fix $A \subset G$ and a Borel set $B \subset Y$ such that $\eta(A) + \nu(B) > 1$ for some $\eta \in \cF_p$.
Assume, for the sake of contradiction, that the complementary set $C = (A^{-1}B)^c$ is not null, and define the 
function $f(g) = \nu(g^{-1}B \cap C)$ for $g \in G$. We note 
that
\[
f(g) \leq \nu(C), \quad \textrm{for all $g \in G$} \qand f(a) = 0, \quad \textrm{for all $a \in A$}.
\]
Furthermore, by Lemma \ref{harmergthm}, we have $\eta(f) = \nu(B) \nu(C)$, and thus (by monotonicity of 
means)
\[
\nu(B) \nu(C) = \eta(f) = \eta(\chi_{A^c} f) \leq  \eta(A^c) \nu(C) = (1-\eta(A))\nu(C).
\]
Since $\nu(C) > 0$, we conclude that $\eta(A) + \nu(B) \leq 1$, which is a contradiction.

\subsection{Proof of Lemma \ref{sumexceed_pmp}}

The proof is almost verbatim the same as for Lemma \ref{sumexceed_harm}, except that when we assume that 
$C = (AB)^c$ is not null, and set $f(g) = \nu(gB \cap C)$ for $g \in G$, we use Scholium \ref{pmpergthm}
instead of Lemma \ref{harmergthm}.

\section{Proofs of Lemma \ref{lemmaleftergLp} and Lemma \ref{lefterg}}

\subsection{Correspondence Principles for measured groups}

Let $(G,p)$ be a countable measured group and let $(X,x_o)$ be a compact metrizable pointed $G$-space. 
Let $T$ be the Bebutov map of $(X,x_o)$, defined as in \eqref{defBebutov}.
We note that since both $\cL_p \subset \ell^\infty(G)^*$ and $\cP_p(X) \subset C(X)^*$ are weak*-closed and 
\emph{convex} subsets of dual Banach spaces, the sets $\cL_p^{\textrm{ext}}$ and 
$\cP_p(X)^{\textrm{ext}}$ of \emph{extreme points} in $\cL_p$ and $\cP_p(X)$ respectively, 
are non-empty by Krein-Milman's Theorem. If we denote by $\cP_G(X)^{\textrm{erg}}$ and 
$\cP_p(X)^{\textrm{erg}}$ the set of \emph{ergodic} elements in $\cP_G(X)$ and $\cP_p(X)$ respectively, 
then, as is well-known (see e.g. \cite{EW} and \cite{BS}), we have
\[
\cP_G(X)^{\textrm{erg}} = \cP_G(X)^{\textrm{ext}}
\qand
\cP_p(X)^{\textrm{erg}} = \cP_p(X)^{\textrm{ext}}.
\]
The following lemma will be useful.
\begin{lemma}
\label{onto}
The adjoint $T^* : \cL_p \ra \cP_p(X)$ is onto, and maps $\cL_p^{\textrm{ext}}$ onto $\cP_p(X)^{\textrm{erg}}$.
\end{lemma}

\begin{proof}
We first prove that $T^*$ is onto. Fix $\nu \in \cP(X)$, and consider the weak*-compact and convex set
\[
\cC_\nu = \big\{ \eta \in \cM \, : \, \eta(T\phi) = \nu(\phi), \quad \textrm{for all $\phi \in C(X)$} \big\}.
\]
We claim that $\cC_\nu$ is non-empty. Indeed, $\nu$ defines a positive and unital functional $\eta_\nu$ on 
the linear subspace $T(C(X)) \subset \ell^\infty(G)$ by $\eta_\nu(T\phi) = \nu(\phi)$. By Hahn-Banach's Theorem,
this functional extends to a norm-one functional $\eta$ on $\ell^\infty(G)$, so it suffices to show that $\eta$ 
is positive (it is unital by construction since $1 \in T(C(X))$ and $\nu$ is a probability measure). Suppose that $f \in \ell^\infty(G)$ is non-negative. Then $\|\|f\|_\infty - f\|_\infty \leq \|f\|_\infty$, and thus
\[
\|f\|_\infty - \eta(f) = \eta(\|f\|_\infty - f) \leq \|\|f\|_\infty -f \|_\infty \leq \|f\|_\infty,
\]
which shows that $\eta(f) \geq 0$, and thus $\eta \in \cC_\nu$. \\

Now suppose that $\nu \in \cP_p(X)$, and fix $\eta \in \cC_\nu$. We claim that $p * \eta \in \cC_\nu$. Indeed, note
that for every $s \in G$ and $\phi \in C(X)$, we have $(s \cdot \eta)(T\phi) = (s \cdot \nu)(\phi)$, and thus
\begin{eqnarray*}
(p * \eta)(T\phi) 
&=& 
\sum_{s \in G} p(s) (s \cdot \eta)(T\phi) = 
\sum_{s \in G} p(s) (s \cdot \nu)(\phi) \\
&=&
\nu(\phi) = \eta(T\phi), \quad \textrm{for all $\phi \in C(X)$}.
\end{eqnarray*}
By Kakutani's Fixed Point Theorem, this implies that there exists $\xi \in \cL_p \cap \cC_\nu$, and thus $T^*$ is onto. \\

Suppose that $\lambda \in \cL_p^{\textrm{ext}}$. We claim that $\nu = T^*\lambda \in \cP_p(X)$ is ergodic. Indeed, by 
Scholium \ref{harmergthmmean}, we know that for every $\eta \in \cF_p$ and for all $\phi_1, \phi_2 \in C(X)$, we have
\[
\int_G \nu((g^{-1} \cdot \phi_1) \phi_2) \, d\eta(g) 
= 
\int_G \lambda((g^{-1} \cdot T\phi_1) T\phi_2) \, d\eta(g) 
=
\lambda(\phi_1) \, \lambda(\phi_2) = \nu(\phi_1) \nu(\phi_2).
\]
By the same approximation argument as in the proof of Lemma \ref{harmwet}, this identity can be extended to all bounded Borel measurable functions $\phi_1$ and $\phi_2$ on $X$. Suppose that $\phi = \chi_B$, where $B \subset X$ is a $G$-invariant Borel set.  The identity above now implies that
\[
\nu(B) = \int_G \nu((g^{-1} \cdot \chi_B) \chi_B) \, d\eta(g) = \nu(B)^2,
\]
and thus $B$ is either null or conull. This shows that $\nu$ must be ergodic. 
\end{proof}

\subsection{Ergodicity for Furstenberg means}

\begin{theorem}[Furstenberg \cite{Fu71}]
For every measured group $(G,p)$ there exists an ergodic $(G,p)$-space $(Z,m)$ such that the bounded,
positive and unital, linear map $P : L^\infty(Z,m) \ra H_r^\infty(G,p)$ defined by
\[
(P \phi)(g) = \int_Z \phi(g \cdot z) \, dm(z),
\]
is an isometric isomorphism of Banach spaces. 
\end{theorem}

We shall refer to $(Z,m)$ as the \textbf{Poisson boundary} of the measured group $(G,p)$, and to $P$ as
its \textbf{Poisson transform}. The main ingredient in the proof of Lemma \ref{lefterg} is the following result, 
which utilize the existence of such a $(G,p)$-space.

\label{subsecerg}
\begin{lemma}
\label{Fpext}
For every measured group $(G,p)$, we have $\cF_p^{\textrm{ext}} \subset \cL_p^{\textrm{ext}}$.
\end{lemma}

\begin{proof}
Let $(Z,m)$ denote the Poisson boundary of $(G,p)$ and let $P : L^\infty(Z,m) \ra H^\infty_r(G,p)$ be its 
Poisson transform. Let $\eta \in \cF_p^{\textrm{ext}}$ and suppose for the sake of contradiction that $\eta$ is \emph{not} an extreme point
in $\cL_p$. This means that we can find \emph{distinct} $\eta_1, \eta_2 \in \cL_p$ so that $\eta = \alpha_1 \eta_1 + \alpha_2 \eta_2$ 
for some \emph{positive} real numbers $\alpha_1, \alpha_2$ with $\alpha_1 + \alpha_2 = 1$. For $i = 1,2$, we define a positive and unital functional $\mu_i \in L^\infty(Z,m)^*$ by
\[
\mu_i(\phi) = \eta_i(P\phi), \quad \textrm{for $\phi \in L^\infty(Z,m)$}.
\]
Since $\eta \in \cF_p$, we have $\eta(P\phi) = m(\phi)$ for every $\phi \in L^\infty(Z,m)$ and thus 
$m = \alpha_1 \mu_1 + \alpha_2 \mu_2$ in the dual space of $L^\infty(Z,m)$. By the Hewitt-Yosida
Theorem, we can write
\[
L^\infty(Z,m)^{*} \cong L^1(Z,m)^{**} \cong L^1(Z,m) \oplus V,
\]
for some linear subspace $V \subset L^1(Z,m)^{**}$, and where we identify $L^1(Z,m)$ with signed measured on $Z$
which are absolutely continuous with respect to $m$. We can then find uniquely determined elements $m_1, m_2 \in L^1(Z,m)$ and $\gamma_1, \gamma_2 \in V$ such that
\[
\mu_i = \beta_i m_i + \gamma_i \lambda_i, \quad \textrm{for $i = 1,2$},
\]
for some non-negative numbers $\beta_i$ and $\gamma_i$ with $\beta_i + \gamma_i = 1$. Since the decomposition 
of $L^1(Z,m)^{**}$ above is clearly invariant under the $G$-action, and $p * \mu_i = \mu_i$ for $i = 1,2$, we must have
\[
p * m_i = m_i \qand p * \lambda_i = \lambda_i, \quad \textrm{for $i = 1,2$}.
\]
Since $m_i << m$, we conclude by Lemma 2.6(2) in \cite{BS} that $m= m_1 = m_2$, and thus
\[
(1 - \alpha_1 \beta_1 - \alpha_2 \beta_2) m = \alpha_1 \gamma_1 \lambda_1 + \alpha_2 \gamma_2 \lambda_2.
\]
Since $m$ does not belong to the linear sub-space spanned by $\lambda_1$ and $\lambda_2$, we conclude that
\[
1 = \alpha_1 \beta_1 + \alpha_2 \beta_2 \qand \alpha_1 \gamma_1 = \alpha_2 \gamma_2 = 0.
\]
Since $a_1$ and $a_2$ are positive, we deduce that $\gamma_1 = \gamma_2 = 0$ and thus $\beta_1 = \beta_2 = 1$. We conclude that $\mu_1 = \mu_2 = m$, and thus $\eta_1, \eta_2 \in \cF_p$. Since $\eta$ is extremal in $\cF_p$, we see that $\eta_1 = \eta_2$, which is a contradiction.
\end{proof}

\subsection{Proofs of Lemma \ref{lemmaleftergLp} and Lemma \ref{lefterg}}

Let $(G,p)$ be a measured group and suppose that $A \subset G$ is either $\cF_p$-large or $\cL_p$-large. 
Since the map $\lambda \mapsto \lambda(A)$ is convex and weak*-continuous on both $\cF_p$ and $\cL_p$,
there exist by Bauer's Maximum Principle (see for instance \cite{AB06}) \textit{extremal} $\lambda_1 \in \cF_p$ and $\lambda_2 \in \cL_p$ such that 
\[
d^*_{\cF_p}(A) = \lambda_1(A) \qand d^*_{\cL_p}(A) = \lambda_2(A).
\]
By Lemma \ref{Fpext}, $\lambda_1$ is also extremal in $\cL_p$, and thus Lemma \ref{lemmaleftergLp} and
Lemma \ref{lefterg} are both consequences of the following lemma. 

\begin{lemma}
\label{leftergLp}
Let $\eta \in \cL_p^{\textrm{ext}}$ and suppose that $A \subset G$ has positive $\eta$-measure. For every $\eps > 0$, there exists a finite set $F \subset G$ such that $\eta(FA) > 1 - \eps$.
\end{lemma}

\begin{proof}
We abuse notation and let $(X,x_o,A)$ denote the Bebutov representation of the set $A$. Let $T$ denote the 
Bebutov map of $(X,x_o)$ as in \eqref{defBebutov}. Since $A \subset X$ is clopen, $FA$ is clopen for every \emph{finite} set $F \subset G$. Hence, if we set $\nu = T^*\eta$, then 
\[
\eta(A_{x_o}) = \nu(A) > 0 \qand \eta(FA_{x_o}) = \nu(FA), \quad \textrm{for every \emph{finite} set $F \subset G$}.
\]
Since $\eta \in \cL_p^{\textrm{ext}}$, we know that 
$\nu$ is ergodic by Lemma \ref{onto}. Fix $\eps > 0$. By ergodicity and our assumption that $\nu(A) > 0$, 
we have $\nu(GA) = 1$, and thus, by $\sigma$-additivity of $\nu$, there exists a finite set $F \subset G$ such that $\nu(FA) > 1 - \eps$. 
We conclude that $\eta(FA_{x_o}) > 1 - \eps$. 
\end{proof}

\subsection{Non-extremality of spherical means}

After seeing the utility of extreme means in $\cL_p$, it is natural to ask how one recognizes such means. Since
the construction of $p$-stationary means on \emph{infinite} groups requires some version of the Axiom of Choice, 
this is probably a very hard question, perhaps lacking a definite answer. In this subsection, we shall focus on the measured group $(\bF_r,\sigma_1)$ and
show that means in $\cS_r$ (which are always $\sigma_1$-stationary) can \emph{never} be extreme in $\cL_{\sigma_1}$. The notation will be as in Subsection \ref{subsec:sphere}. \\

Let us first deal with the (somewhat degenerate) case $r = 1$, where $\bF_1 \cong \bZ$ and $a_1$ corresponds to
the generator $1$. We note that
\[
\sigma_k = \frac{1}{2}(\delta_{k} + \delta_{-k}), \quad \textrm{for $k \geq 1$}, 
\]
and thus each $\lambda \in \cS_1$ is a weak*-cluster point of the sequence $(\lambda_n)$ defined by
\[
\lambda_n(B) = \frac{|B \cap [-n,n]|}{2n+1}, \quad \textrm{for $B \subset \bZ$}. 
\]
Let $T = [1,\infty) \subset \bZ$, and note that $T$ is a sub-semigroup of $\bZ$ such that $\bZ = T \sqcup (-T) \sqcup \{0\}$. One readily checks that if $\lambda \in \cS_1$, then $\lambda(T) = \lambda(-T) = 1/2$, and 
\[
\lambda = \frac{1}{2} \lambda_{+} + \frac{1}{2} \lambda_{-}(T), \quad \textrm{where $\lambda_{\pm}(\cdot) = 2 \lambda( \cdot \cap \pm T)$}.
\]
Furthermore, since for every $n \in \bZ$, the symmetric difference $T \Delta (T \pm n)$ has zero $\lambda$-measure, 
we conclude that $\lambda_{+}$ and $\lambda_{-}$ are distinct element in $\cL_{\sigma_1} = \cL_{\bF_1}$. In particular, $\lambda$ is not extreme in $\cL_{\sigma_1}$. \\

The case when $r \geq 2$ is slightly more involved. Let us fix $\lambda \in \cS_r$ and a homomorphism 
$\tau : \bF_r \ra \bZ$ whose kernel has zero $\lambda$-measure. Define 
\[
T = \big\{ g \in G \, : \, \tau(g) \geq 1 \big\} \subset \bF_r,
\]
and note that $T$ is a sub-semigroup of $\bF_r$, with the property that $T \Delta gT$ has zero $\lambda$-measure
for every $g \in G$, and $\bF_r = T \sqcup T^{-1} \sqcup \ker \tau$. We conclude that $\lambda(T) = \lambda(T^{-1}) = 1/2$, and if we set
\[
\lambda_{+}(B) = 2 \lambda(B \cap T) \qand \lambda_{-}(B) = 2\lambda(B \cap T^{-1}), \quad \textrm{for $B \subset \bF_r$},
\]
then $\lambda$ can be written as a non-trivial convex combination of $\lambda_{+}$ and $\lambda_{-}$, which are 
distinct elements in $\cL_{\sigma_1}$ (we leave the details to the reader). We conclude that $\lambda$ is not 
extreme in $\cL_{\sigma_1}$. \\

It might be worth emphasizing here that although the main results in Subsection \ref{subsec:sphere} are stated 
in terms of spherical densities, we do not know of proofs of these results which do not utilize the fact (Lemma \ref{exFp}) that the larger set $\cF_{\sigma_1}$ does contain extreme elements in $\cL_p$.


\section{SAT*-spaces and the proof of Theorem \ref{piecewiseFol}}
\label{sec:prfpwFol}
Let $G$ be a countable group and suppose that $(Z,m)$ is a standard Borel probability measure
space, equipped with a measure-class preserving action of $G$ by bi-measurable maps. We say
that $(Z,m)$ is \textbf{SAT} (strongly approximate transitive) if for every Borel set $B \subset Z$
which is not conull, and $\eps > 0$, there exists $g \in G$ such that $m(g^{-1}B) < \eps$. In 
other words, $(Z,m)$ is far away from being a p.m.p.\ $G$-space. This notion was first introduced
by Jaworski in \cite{Jaw}. Finally, we say that $(Z,m)$ is 
\textbf{SAT*} if it is SAT and if for every $g \in G$, the Radon-Nikodym derivative $\frac{d(g \cdot m)}{dm}$ 
is essentially bounded on $Z$. This notion was first introduced by Kaimanovich in \cite{KaSAT}. 

We claim that the Poisson boundary of a measured group is always SAT. Indeed, if $(G,p)$ is a
measured group, $(Z,m)$ its Poisson boundary and $P$ its Poisson transform, then for every 
Borel set $B \subset Z$ which is not conull, we have $\|\chi_{B^c}\|_{L^\infty} = 1$ and $P\chi_{B^c}(g) = 1 - m(g^{-1}B)$
for all $g \in G$. Since $P$ is an isometric map, we conclude that we must have $\|P\chi_{B^c}\|_{H^\infty_r} = 1$ as well, 
which readily shows that $m(g^{-1}B)$ can be made arbitrarily small upon varying $g$. 

We claim that if $(Y,\nu)$ is any $(G,p)$-space, then the Radon-Nikodym derivative $\frac{ds\nu}{d\nu}$ is always
essentially bounded for every $s \in G$ (in particular, the Poisson boundary of the measured group $(G,p)$ is SAT*). Indeed, note that by $p$-stationarity of $\nu$, we have
\[
\sum_{g \in G} p^{*n}(g) \frac{dg\nu}{d\nu}(y) = 1, \quad \textrm{for $\nu$-a.e. $y$}.
\]
Fix $s \in G$. By admissibility of $p$, there exists an integer $n$ such that $p^{*n}(s) > 0$, and thus
\[
p^{*n}(s) \frac{ds\nu}{d\nu}(y) \leq \sum_{g \in G} p^{*n}(g) \frac{dg\nu}{d\nu}(y) = 1, 
\]
for $\nu$-a.e. $y$, and thus $\frac{ds\nu}{d\nu}(y) \leq 1/p^{*n}(s)$ almost everywhere. \\

Our main lemma in this section reads as follows.

\begin{lemma}
\label{FsB}
Let $(Z,m)$ be a SAT* $G$-space and fix $\eps > 0$. For every non-conull Borel set $B \subset Y$ and for every 
\emph{finite} set $F \subset G$, there exists $s \in G$ such that
\[
m(FsB) < \eps.
\]
In particular, for every $\eps > 0$ and non-conull Borel set $B \subset Y$ there exists a thick set $A \subset G$
such that $m(AB) < \eps$.
\end{lemma}

\begin{proof}
Let $B \subset Y$ be a non-conull Borel set and let $F \subset G$ be a finite set. For every $s \in G$, we have
\[
m(FsB) \leq |F| \, \max_{f \in F} \big\| \frac{df^{-1}m}{dm}\big\|_\infty \, m(sB).
\]
The right hand side is finite since the $G$-space $(Z,m)$ is SAT*. Fix $\eps > 0$. Since $(Z,m)$ is SAT, we 
can now find $s \in G$ such that 
\[
m(sB) < \frac{\eps}{ |F| \, \max_{f \in F} \big\| \frac{df^{-1}m}{dm}\big\|_\infty},
\]
which combined with the first inequality finishes the proof of the first assertion. Let $(F_n)$ be an increasing exhaustion of $G$ by finite sets, and 
fix a decreasing sequence $(\eps_n)$ of positive numbers whose sum is less than $\eps$. For every $n$ we can
find $s_n \in G$ such that $m(F_n s_n B) < \eps_n$. Let $A$ denote the union of all the sets $F_n s_n$. 
By construction, $A$ is thick, and 
\[
m(AB) \leq \sum_{n} m(F_n s_n B) < \sum_{n} \eps_n < \eps. 
\]
\end{proof}

\begin{corollary}
\label{leftthicksets}
Let $(Z,m)$ be a SAT*-space. For every Borel set $B \subset Y$ which is not null and for $\nu$-almost every $y$ in $Y$, the set $B_y$ is left thick.
\end{corollary}

\begin{proof}
Let $(F_n)$ be an increasing exhaustion of $G$ with finite subsets and let $(\eps_n)$ be a decreasing sequence 
of positive numbers whose sum does not exceed $\eps$. Since $B^c$ is not conull, we can, by Lemma \ref{FsB}, find a sequence $(s_n)$ 
in $G$ such that $\nu(F_n^{-1}s_n^{-1}B^c) < \eps_n$ for every $n$. In other words, the set
\[
E_n = (F_n^{-1}s_n^{-1}B^c)^c = \bigcap_{f \in F_n} f^{-1}s_n^{-1} B = \big\{ y \in Y \, : \, s_n F_n \subset B_y \big\} \subset Y
\]
has measure at least $1-\eps_n$. Let $Y_\eps$ denote the intersection of all of the sets $E_n$. One readily checks that
\[
\nu(Y_\eps) \geq 1 - \sum_{n} \eps_n \geq 1 - \eps,
\]
and by construction, $B_y$ is left thick for every $y \in Y_\eps$. Since $\eps > 0$ is arbitrary, we conclude that $B_y$
is left thick for $\nu$-almost every $y$.
\end{proof}

Let us now establish a functional variant of Lemma \ref{FsB}.

\begin{lemma}
\label{lemmaT}
Let $(Z,m)$ be a SAT* $G$-space. For every non-constant $\phi \in L^\infty(Z,m)$ with non-negative integral, the
set
\[
T = \big\{ g \in G \, : \, \int_Z \phi(gz) \, dm(z) > 0 \big\}
\]
is left thick. 
\end{lemma}

\begin{proof}
Since $\phi$ is non-constant with non-negative integral, there exists $\delta > 0$ such that the set
\[
E = \big\{ z \in Z \, : \, \phi(z) \geq \delta \big\}
\]
has positive measure. Choose $\eps > 0$ so that
\[
\delta > \frac{\eps}{1-\eps} \|\phi\|_\infty. 
\]
Let $F \subset G$ be a finite set. By Lemma \ref{FsB} we can find $s \in G$ such that 
$m(F^{-1}s^{-1}E^c) < \eps$, or equivalently, 
\[
m\Big( \bigcap_{f \in F} f^{-1}s^{-1}E \Big) \geq 1-\eps.
\]
Let $E'$ denote the intersection in the last inequality. Note that for every $z \in E'$ and $f \in F$, we have 
$sfz \in E$. Hence, for every $f \in F$,
\begin{eqnarray*}
\int_Z \phi(sfz) \, dm(z) 
&=& 
\int_{E'} \phi(sfz) \, dm(z) + \int_{(E')^c} \phi(sfz) \, dm(z) \\
&\geq &
\delta m(E') - \|\phi\|_\infty (1-m(E')) \\
&\geq & 
\delta (1-\eps) - \|\phi\|_\infty \eps > 0,
\end{eqnarray*}
by our choice of $\eps > 0$. This argument shows that $sF \subset T$. Since $F$ was
arbitrary, we conclude that $T$ is left thick. 
\end{proof}

\begin{corollary}
\label{corharmthick}
For every non-constant $f \in H^\infty_r(G,p)$ with $f(e) \geq 0$, the set $V = \big\{ g \in G \, : \, f(g) \geq 0 \big\}$ is 
left thick. 
\end{corollary}

\begin{proof}
Let $(Z,m)$ denote the Poisson boundary of $(G,p)$. Recall that $(Z,m)$ is a SAT* $G$-space and for 
every $f \in H^\infty_r(G,p)$, there exists a unique $\phi \in L^\infty(Z,m)$ such that
\[
f(g) = \int_Z \phi(gz) \, dm(z), \quad \textrm{for all $g \in G$}.
\]
Furthermore, if $f$ is non-constant, then so is $\phi$, and if $f(e) \geq 0$, then $\phi$ has non-negative 
integral. Hence, under our assumptions, the conditions on $\phi$ in Lemma \ref{lemmaT} are satisfied and
since
\[
V  \supset \{ g \in G \, : \, \int_Z \phi(gz) \, dm(z) > 0 \big\},
\] 
we conclude that the set $V$ is left thick. 
\end{proof}

\begin{corollary}
\label{corS}
Let $(Y,\nu)$ be a $(G,p)$-space and let $B \subset Y$ be a Borel set. Then the set
\[
S = \big\{ g \in G \, : \, \nu(gB) \geq \nu(B) \big\}
\]
is thick. 
\end{corollary}

\begin{proof}
Define $f(g) = \nu(g^{-1}B) - \nu(B)$. Since $\nu$ is $p$-stationary, $f \in H^\infty_r(G,p)$ and $f(e) = 0$. If 
$f$ is identically equal to zero, then $S = G$ and the corollary is trivial. Hence we may assume that $f$ is 
non-constant. Corollary \ref{corharmthick} now implies that the set
\[
V = \{ g \in G \, : \, f(g) \geq 0 \big\}
\]
is left thick. Since $S = V^{-1}$, we conclude that $S$ is thick. 
\end{proof}

\begin{corollary}
\label{corU}
Let $(Y,\nu)$ be a $(G,p)$-space and let $B, C \subset Y$ be Borel sets. Suppose that 
the inequality $\nu(B) + \nu(C) > 1$ holds. Then the set
\[
U = \big\{ g \in G \, : \, \nu(C \cap gB) > 0 \big\} 
\]
is thick. 
\end{corollary}

\begin{proof}
Since $\nu(B) + \nu(C) > 1$, we have $\nu(C \cap gB) > 0$ whenever $\nu(gB) \geq \nu(B)$, and thus
\[
U \supset \big\{ g \in G \, : \, \nu(gB) \geq \nu(B) \big\}.
\]
By Corollary \ref{corS}, the set on the right hand side is thick, and thus so is $U$.
\end{proof}

\subsection{Proof of Theorem \ref{piecewiseFol}}

Let $(G,p)$ be a countable measured group and let $A \subset G$ be a $\cL_p$-large set. Fix 
$\lambda \in \cL_p^{\textrm{ext}}$ such that $\lambda(A) > 0$. We abuse notation and
denote by $(X,A,x_o)$ the Bebutov triple of $A$, and by $T$ its Bebutov map as in 
\eqref{defBebutov}. Set $\nu = T^*\lambda$ and note that $\lambda(A_{x_o}) = \nu(A) > 0$.
Also, $\nu \in \cP_p^{\textrm{erg}}(X)$ by Lemma \ref{onto}. 

By Lemma \ref{leftergLp}, we can 
find a finite set $F \subset G$ such that 
\[
\lambda(FA_{x_o}) = \nu(FA) > 1 - \nu(A).
\]
We note that
\[
FA_{x_o}A_{x_o}^{-1} \supset \big\{ g \in G \, : \, \lambda((FA \cap gA)_{x_o}) > 0 \big\} 
= \big\{ g \in G \,: \, \nu(FA \cap gA) > 0 \big\}.
\]
Since $\nu(FA) + \nu(A) > 1$, we conclude by Corollary \ref{corU} that $FA_{x_o}A_{x_o}^{-1}$ is thick,
and thus $A_{x_o}A_{x_o}^{-1}$ is piecewise syndetic. \\

Finally, note that we also have
\[
A_{x_o}A_{x_o}^{-1}F^{-1} \supset \big\{ g \in G \, : \, \lambda((A \cap gFA)_{x_o}) > 0 \big\} 
= \big\{ g \in G \,: \, \nu(A \cap gFA) > 0 \big\}.
\]
Since $\nu(A) + \nu(FA) > 1$, we conclude again by Corollary \ref{corU} that $A_{x_o}A_{x_o}^{-1}F^{-1}$ is thick,
and thus the difference set $A_{x_o}A_{x_o}^{-1}$ is also piecewise \emph{left} syndetic. 

\section{Proof of Theorem \ref{notLiouville}}
\label{sec:prfthmnotliou}

Let $Y$ be a locally compact and metrizable space, equipped with a \textsc{minimal} action (all orbits are dense) 
of a countable group $G$ by homeomorphisms. Let $p$ be an admissible (not necessarily symmetric) 
probability measure on $G$, and suppose that there exists a \textsc{unique} $p$-stationary (ergodic) Borel 
probability measure on $Y$. We note that if $G$ is \emph{amenable} and $Y$ is \emph{compact}, then there is 
always at least one \emph{$G$-invariant} probability measure on $Y$, and thus $(G,p)$ must be Liouville. On the
other hand, if $Y$ is \emph{not} compact, then this implication does not need to hold as Example \ref{affine} below 
clearly demonstrates.

Let us further assume that the $(G,p)$-space $(Y,\nu)$ is SAT*, so that the results in the previous section apply. 
We shall review some examples below.
Fix a compact set $C \subset Y$ with empty interior and positive $\nu$-measure, and let $(F_n)$ be an exhaustion of $G$ by finite subsets. Then $D = C^c$ is an open subset of $Y$ which is not conull, and by Lemma \ref{FsB} we can
find a sequence $(s_n)$ in $G$ such that
\[
\nu(AD) < \frac{1}{2}, \quad \textrm{where $A = \bigcup_{n} F_n s_n$}.
\]
We note that $A$ is thick, and thus
\[
B := (AD)^c = \bigcap_{a \in A} aC
\]
is a closed subset of $Y$ with positive $\nu$-measure. By construction, we have $A^{-1}B \subset C$. 
In particular, for every $y \in Y$ we have $A^{-1}B_y \subset C_y$. By Corollary \ref{leftthicksets}, $B_y$ 
is left thick for $\nu$-almost every $y \in Y$. \\

The following two lemmas, applied to the examples below, will now finish the proof of Theorem \ref{notLiouville}.

\begin{lemma}
Suppose that $B \subset Y$ is Borel. Then there exists a $\nu$-conull subset $Y' \subset Y$ such that 
$d^*_{\cF_p}(B_y) \geq \nu(B)$ for all $y \in Y$.
\end{lemma}

\begin{proof}
Let $(\beta_n)$ be as in \eqref{exFp}. By the pointwise Ergodic Theorem for random walks (see e.g. \cite{Ka}),
there exists a conull subset $Y' \subset Y$ such that
\[
\beta_n(B_y) = \frac{1}{n} \sum_{k=0}^{n-1} \Big( \sum_{s \in G} \chi_B(s \cdot x) \Big) \, p^{*n}(s) \ra \nu(B).
\]
Since any weak*-cluster point of $(\beta_n)$ must belong to $\cF_p$, we conclude that $d^*_{\cF_p}(B_y) \geq \nu(B)$
for all $y \in Y'$.
\end{proof}

\begin{lemma}
If $C \subset Y$ is a compact set with empty interior, then $C_y$ is not piecewise syndetic for any $y \in Y$. If $Y$ is not compact, it suffices to assume that $C$ is compact. 
\end{lemma}

\begin{proof}
If $Y$ is non-compact, let $Y^\infty$ denote the one-point compactification of $Y$, and note that $Y^\infty$ is a compact $G$-space with exactly two \emph{ergodic} Borel probability measures, namely $\nu$ and $\delta_\infty$, where $\infty$ is the added point at infinity. Since $C$ is compact in $Y$ it is closed in $Y^\infty$. If $Y$ is compact, 
we keep it as it is. 

Suppose that there exists $y \in Y^\infty$ such that $C_y$ \emph{is} piecewise syndetic, that is to say, suppose that there exists a finite set 
$F \subset C$ such that $FC_y$ is thick. It is not hard to see that this implies that there exists an extremal 
$\eta \in \cL_p$ such that $\eta(FC_y) = 1$. Since $C$ is compact, so is $FC$, and thus Lemma \ref{onto} and 
Lemma \ref{open} show that there exists an $p$-stationary probability measure $\mu$ on $Y^\infty$ (or $Y$, if it is 
compact) such that $\mu(FC) = 1$. 

We assume that $\nu$ is the \emph{unique} $p$-stationary measure on $Y$. Hence, if $Y$ is not compact, then 
$\mu$ is a convex combination of $\nu$ and $\delta_\infty$, and thus (since $\delta_\infty(C) = 0$), we must have $\mu(FC) = \nu(FC)$. In either case ($Y$ compact or not), we conclude 
that $FC = \supp(\nu) = Y$ (since the action is minimal and $\supp(\nu)$ is a closed $G$-invariant set). Since $C$ is
assumed to have empty interior, this is impossible by Baire's Category Theorem. 

Finally, note that $FC$ is always compact, and thus $FC = Y$ can never hold if $Y$ is non-compact.
\end{proof}

Let us consider two examples of measured groups equipped with actions which satisfy the requirements above.
We stress that we do not know at this point whether in fact \emph{every} non-Liouville measured group admits such actions.

\begin{example}
\label{free}
Let $\bF_r$, $S$ and $\sigma_1$ be as in Subsection \ref{subsec:sphere}, and let $Y$ denote the space of all
infinite reduced words in $S$, viewed as a closed subset of the compact metrizable space $S^\bN$. We note 
that $\bF_r$ acts on $Y$ by concatenation and subsequent reductions. It is proved in \cite{Fu71} that there is a 
unique $\sigma_1$-stationary Borel probability measure $\nu$ on $Y$, and $(Y,\nu)$ is the Poisson 
boundary of $(G,\sigma_1)$. In particular, this is a SAT*-space.
\end{example}

\begin{example}
\label{affine}
Let $\Aff(\bR)$ denote the group of all affine maps on $\bR$. We note that this group is isomorphic to the semi-direct product $\bR \rtimes \bR^*$, and is thus two-step solvable (hence amenable). Let $p$ be a probability measure on a countable 
subgroup $G$ of $\Aff(\bR)$ which acts minimally on $\bR$. Under mild conditions (which forbid symmetry of $p$), see e.g. \cite{Bro}, it can be 
shown that the $G$-action on $\bR$ admits a unique $p$-stationary probability measure $\nu$ such that 
$(\bR,\nu)$ is SAT*.
\end{example}

\section*{Appendix I: Failure of syndeticity for differences of $\cL_p$-large sets}

The aim of this appendix is to isolate some properties of a measured group $(G,p)$ which will ensure that 
there are $\cF_p$-large subsets whose difference sets are \emph{not} syndetic. However, such difference
sets are always \emph{piecewise} syndetic by Theorem \ref{piecewiseFol}. \\

We begin by defining detaching actions.

\begin{definition}
\label{def}
A locally compact $G$-space $X$ is called \textbf{detaching} if for every pair of \emph{proper} compact
subsets $K, L \subset X$, there exists $s \in G$ such that $K \cap sL = \emptyset$.
\end{definition}
 
\begin{remark} 
If $X$ is compact, then detaching actions can be equivalently defined as follows: For every non-empty open set $U \subset X$ and proper closed set $L \subset X$, we can find $s \in G$ such that $gL \subset U$. We stress that \emph{compact} detaching 
$G$-spaces $X$ are referred to as \textbf{extremely proximal} in the literature. Properties
of such extremely proximal actions are discussed in detail in \cite{Gl}, where it is shown, among other things, that 
a countable group which admits a non-trivial extremely proximal action must contain a non-cyclic free group (in particular, it
cannot be amenable). However, as we shall see below, there are examples \textsc{amenable} groups 
which admit detaching actions on (non-compact) locally compact spaces.  
\end{remark}

\begin{theorem}
\label{deta}
Let $X$ be a locally compact and second countable detaching $G$-space, and suppose that there exists an ergodic $p$-stationary 
Borel probability measure $\nu$ on $X$ with full support. Then, for every compact set $A \subset X$ with empty interior, there exists $\eta \in \cF_p$ and a $\nu$-conull subset $X' \subset X$ such that 
\[
\eta(A_x) = \nu(A) \qand A_x A_x^{-1} \enskip \textrm{is not syndetic},
\]
for all $x \in X'$. If $X$ is non-compact, then the condition that $A$ has empty interior is not needed. 
\end{theorem}

\begin{remark} 
Note that a locally compact detaching $G$-space $X$ cannot admit a \emph{$G$-invariant} Borel probability 
measure $\nu$. Indeed, by regularity of such a $\nu$, we can find a compact set $K \subset X$ whose measure
exceeds $1/2$. Since $X$ is detaching, we can find $s \in G$ such that $K \cap sK$ is empty. However, since $\nu$
is assumed to be $G$-invariant, this implies that
\[
\nu(K \cup sK) = \nu(K) + \nu(sK) - \nu(K \cap sK) = 2\nu(K) > 1,
\] 
which is a contradiction. 
\end{remark}

\begin{example}[Free groups acting on boundaries of trees]
\label{example1}
Let $\bF_r$, $S$ and $\sigma_1$ be as in Subsection \ref{subsec:sphere}, and let $Y$ denote the space of all
infinite reduced words in $S$, viewed as a closed subset of the compact metrizable space $S^\bN$. We note 
that $\bF_r$ acts on $X$ by concatenation and subsequent reductions. It is proved in \cite{Fu71} that this action
is extremely proximal (hence detaching) and that there is a unique $p$-stationary probability measure $\nu$ on 
$X$. Furthermore, the action is minimal and $\nu$ has full support.
\end{example}

\begin{example}[Affine groups acting on the real line]
\label{example2}
Let $G < \Aff(\bR)$ be a countable group and let $p$ be an admissible probability measure on $G$. It 
is proved in \cite{Bro} that under mild conditions, there is a unique (ergodic) $p$-stationary probability 
measure on $\bR$. One readily checks that $G$-action on $\bR$ is detaching provided that $G$ contains
an infinite subgroup of translations.
\end{example}

\begin{proof}[Proof of Theorem \ref{deta}]
Let $A \subset X$ be a proper compact set and let $F \subset G$ be a finite set. If $X$ is non-compact, then 
$FA$ must still be a proper subset of $X$, and since $X$ is a detaching $G$-space, we can find $s \in G$ such 
that $FA \cap sA$ is empty. In particular, $s \notin FA_x A_x^{-1}$ for all $x \in X$. Since $F$ is arbitrary, we 
conclude that $A_x A_x^{-1}$ cannot be syndetic for any $x$. If $X$ is compact and $A$ has empty interior, then 
$FA$ must still be a proper subset of $X$ by Baire's Category Theorem, and the same argument as above implies 
that $A_x A_x^{-1}$ cannot be syndetic for any $x \in X$. 

It remains to estimate the "size" of $A_x$. Let $(\beta_n)$ be the sequence of probability measures on $G$ defined
in \eqref{exFp}, and note that
\[
\beta_n(A_x) = \frac{1}{n} \sum_{k=0}^{n-1} \Big( \sum_{s \in G} \chi_A(s \cdot x) \, p^{*k}(s)\Big) 
\]
for all $x \in X$. Since $\nu$ is an ergodic $p$-stationary measure on $X$, the pointwise Ergodic Theorem for random walks  (see e.g. \cite{Ka}) shows that there exists a $\nu$-conull Borel set $X' \subset X$ such that 
$\beta_n(A_x) \ra \nu(A)$ forall 
$x \in X'$. Since any weak*-cluster point of $(\beta_n)$ belongs to $\cF_p$, we conclude that there exists $\eta \in \cF_p$ such that $\eta(A_x) = \nu(A)$ for all $x \in X'$.
\end{proof}

In particular, we get:

\begin{corollary}
Let $(G,p)$ be a measured group and suppose that there exists a locally compact detaching $G$-space with an 
ergodic $p$-stationary Borel probability measure. Then there exists a $\cF_p$-large set $A \subset G$ such that 
$AA^{-1}$ is \emph{not} syndetic. 
\end{corollary}

\end{document}